\documentclass[11pt]{amsart}
\usepackage{amsthm}
\usepackage{amscd}
\usepackage{graphicx}
\usepackage{amssymb}
\usepackage{epstopdf}
\theoremstyle{plain}
\newtheorem{Thm}{Theorem}[section]

\newtheorem{Prop}[Thm]{Proposition}

\newtheorem{Lem}[Thm]{Lemma}

\theoremstyle{definition}

\newtheorem{Rem}[Thm]{Remark}
 
\numberwithin{equation}{section}

\title{Semi-orthogonal decomposition and smoothing}
\author{Yujiro Kawamata}
\usepackage{fancyhdr}
\begin{document}
\maketitle

\tableofcontents

\begin{abstract}
We investigate the behavior of semi-orthogonal decompositions of bounded derived categories of singular varieties
under flat deformations to smooth varieties. 
We consider a $\mathbf Q$-Gorenstein smoothing of a surface with a quotient singularity, and 
prove that a pretilting sheaf, which is constructed from a non-commutative deformation of a divisorial sheaf and  
weakly generates a semi-orthogonal component of a bounded derived category, 
deforms to a direct sum of exceptional vector bundles which are mutually totally orthogonal.
\end{abstract}

%14J60, 14B07, 14E30.

%%%%%%%%%%%%%%%%%%%%%%%%%%%%%%%%%%%%%%
%%%%%%%%%%%%%%%%%%%%%%%%%%%%%%%%%%%%%%
%%%%%%%%%%%%%%%%%%%%%%%%%%%%%%%%%%%%%%
\section{Introduction}

The derived category of a smooth projective variety behaves nicely, but that of a singular projective variety $X$
does not.
For example, $\mathbb R\text{Hom}(A,B)$ may be unbounded for objects $A,B$ 
of a bounded derived category of coherent sheaves $D^b(X) = D^b(\text{coh}(X))$.
In order to study the latter case, we can use a resolution of singularities $f: \tilde X \to X$, and 
study $D^b(X)$ using $D^b(\tilde X)$.
In this paper, we try another way and consider a smoothing of $X$, a deformation 
to a smooth variety $Y$.

A coherent sheaf $F$ on a normal complete variety $X$ is said to be {\em pretilting} 
if all higher self-extensions vanish: $\text{Ext}^p(F,F) \cong 0$ for $p > 0$.  
If $F$ comes from a {\em non-commutative (NC) deformation} of a simple sheaf, then 
$F$ weakly generates a subcategory in $D^b(X)$ (in the sense explained in \S 2.1)
which is equivalent to a bounded derived category $D^b(R) = D^b(\text{mod}(R))$ of finitely generated right $R$-modules 
over a finite dimensional associative algebra $R = \text{End}(F)$, the parameter algebra of the NC deformation, 
and $D^b(X)$ has a corresponding semi-orthogonal decomposition (cf. \cite{NC multi}).  
We are interested in the behavior of $F, R$ and $D^b(X)$ 
when $X$ is deformed to a smooth projective variety $Y$.  

A pretilting object is a generalization of an {\em exceptional object}, where $R \cong k$ is the base field.
A certain important smooth projective variety such as a projective space has a full exceptional collection, 
i.e., the derived category $D^b(X)$ is generated by a sequence of exceptional objects 
$(e_1,\dots,e_m)$ which are semi-orthogonal: $\mathbb R\text{Hom}(e_i,e_j) \cong 0$ for $i > j$. 
But it seems that the derived category of a singular variety has never such a collection.
Instead they are sometimes semi-orthogonally decomposed into subcategories weakly generated by pretilting objects.
The endomorphism ring $R = \text{End}(F)$ 
of a pretilting object is an associative algebra which is not necessarily commutative.

We consider the case of a normal surface $X$ which has a {\em $\mathbf Q$-Gorenstein smoothing}.
It is a special type of deformation which we encounter in the minimal model program (\cite{KSB}).
We consider in this paper the behavior of derived categories under such deformations.
The derived category of a certain singular surface has a pretilting sheaf 
which arises as a versal non-commutative deformation of a divisorial sheaf, a reflexive sheaf of rank $1$.
We investigate the behavior of such a pretilting sheaf
under the $\mathbf Q$-Gorenstein smoothing, and
prove that it deforms to a direct sum of exceptional sheaves which are mutually totally orthogonal.

The main theorem of this paper is the following:

\begin{Thm}[Theorems \ref{Hacking'}, \ref{direct sum'}]
Let $X$ be a normal projective surface (variety of dimension $2$) 
such that $H^p(X,\mathcal{O}_X)= 0$ for $p > 0$.
Assume the following conditions:

(a) There is a quotient singularity $P \in X$ of type $\frac 1{r^2s}(1, ars-1)$ for positive integers $a,r,s$ such that 
$0 < a < r$ and $(r,a) = 1$.
Let $D_{P,1}$ and $D_{P.2}$ be coordinate divisors in an analytic neighborhood of $P$ corresponding to the
weights $1$ and $ars-1$.

(b) There exists a divisorial sheaf $A = \mathcal{O}_X(-D)$ on $X$ for a Weil divisor $D$ such that 
$A \cong \mathcal{O}(-D_{P,1})$ in an analytic neighborhood of $P$.
Moreover, either $D$ or $D - K_X$ is a Cartier divisor at each point other than $P$, 
where the choice of $D$ or $D-K_X$ depends on the point.

(c) There is a projective flat deformation $f: \mathcal X \to \Delta$ over a disk $\Delta$ such that $X \cong f^{-1}(0)$,  
$f^{-1}(t)$ is smooth for $t \ne 0$, and that $\mathcal X$ is $\mathbf Q$-Gorenstein at $P$, 
i.e., $f$ is a $\mathbf Q$-Gorenstein smoothing at $P$. 

Then, after replacing $\Delta$ by a smaller disk around $0$ and after a finite base change 
by taking roots of the coordinate $t$, 
there exist on $\mathcal X$ maximally Cohen-Macaulay sheaves $\mathcal E_1,\dots \mathcal E_s$ of rank $r$ 
as well as a coherent sheaf $\mathcal F$ of rank $r^2s$ 
which is locally free at $P$ and locally free or dual free at other points
(i.e., $\mathcal F$ is locally isomorphic to either $\mathcal{O}_{\mathcal X}^{\oplus r^2s}$ or 
$\omega_{\mathcal X}^{\oplus r^2s}$ at each point depending on the point), 
which satisfy the following conditions:

\begin{enumerate}
\item $\mathcal E_i \otimes \mathcal{O}_X \cong A^{\oplus r}$ for all $i$.

\item $E_i := \mathcal E_i \otimes \mathcal{O}_Y$ are exceptional vector bundles on $Y = f^{-1}(t)$ for $t \ne 0$, 
which are mutually orthogonal, i.e., 
\[
\mathbb R\text{Hom}_Y(E_i,E_j) := \bigoplus_{p=0}^2 \text{Ext}^p(E_i,E_j)[-p] \cong 0
\]
for $i \ne j$.

\item $F:= \mathcal F \otimes \mathcal{O}_X$ is constructed as a versal non-commutative deformation 
(see \S 3.1 or [14] for the definition of versal NC deformations) 
of the sheaf $A$ on $X$, and is a pretilting sheaf.

\item $\mathcal F \otimes \mathcal{O}_Y \cong \bigoplus_{i=1}^s E_i^{\oplus r}$.
In particular $\text{End}(\mathcal F \otimes \mathcal{O}_Y) \cong Mat(k,r)^{\times s}$.
\end{enumerate}
\end{Thm}

We note that the singularity of type $\frac 1{r^2s}(1, ars-1)$ appeared naturally when we applied 
the minimal model theory of $3$-folds to the degeneration of surfaces (\cite{KSB}).

$D^b(X)$ and $D^b(Y)$ have semi-orthogonal components
$\overline{\langle F \rangle} \cong D^b(R)$ with $R = \text{End}(F)$ and
$\langle E_i \rangle \cong D^b(k)$ for $1 \le i \le s$, respectively, which are related by the deformation
(see \S 2.1 for the notation).
We note that the associative algebra $R$ is calculated by \cite{KK} using \cite{HiP} and is called a 
{\em Kalck-Karmazyn algebra} (Theorem \ref{KKS}).

The assertions (1) and (2) are generalizations of a result of Hacking \cite{H} 
where the case $s = 1$ ({\em Wahl singularity}) is treated.
The sheaf $\mathcal E_1$ here is the dual of a reflexive sheaf
$\mathcal E$ in \cite{H} Theorem 1.1 which satisfies that the double dual of the restriction 
$(\mathcal E \otimes \mathcal O_X)^{**}$ 
is isomorphic to the dual of $A^{\oplus r}$.
We note that our theorem does not need to take a double dual on $X$.
It also gives a natural explanation of Hacking's exceptional vector bundles in terms of 
a semi-universal non-commutative deformation which is unique up to isomorphisms.
It is also remarkable that the reflexive sheaves $\mathcal E_i$ are mutually orthogonal on the generic fiber but 
reduced to the same sheaf on the special fiber. 

We explain the plan of this paper.
We start with the background material in \S 2 on $\mathbf Q$-Gorenstein smoothing, exceptional vector bundles, 
semi-orthogonal decompositions, pretilting objects, and a motivating example of a weighted projective plane
$\mathbf P(1,1,4)$.
In \S 3, we recall the theory of non-commutative deformations of divisorial sheaves on a surface 
with quotient singularities, and explain a semi-orthogonal decomposition of the derived category for a weighted projective plane 
by Karmazyn-Kuznetsov-Shinder \cite{KKS}. 
We prove the main theorem in \S 4 ($s=1$ case) and in \S5 (general case).
The mutual orthogonality of the $E_i$ is proved by using flops between different {\em crepant simultaneous partial 
resolutions}; the semi-orthogonality and flops imply the full orthogonality.
Then in \S 6, we consider weighted projective planes as an example.

After the first version of this paper is submitted to arXiv, the author was informed 
that a generalization of Hacking's results 
to higher Milnor numbers was already considered in \cite{Cho}, and 
the same results as (1) and (2) of the above theorem were already obtained there
except that taking the double dual was still needed in (1).
The proofs of the orthogonality in (2) are different in the sense that 
our argument uses flops of 3-folds and more geometric.
The author would like to thank Professor Yonghwa Cho for this information.

The author would also like to thank the referee for numerous and helpful suggestions for improving this paper.
The proof of the fullness in Theorem \ref{Pa1a2a3} is due to him.
He would also like to thank NCTS of National Taiwan University where 
the work was partly done while the author visited there.
This work is partly supported by JSPS Kakenhi 16H02141 and 21H00970.

We work over the base field $k = \mathbf C$.

%%%%%%%%%%%%%%%%%%%%%%%
%%%%%%%%%%%%%%%%%%%%%%%
%%%%%%%%%%%%%%%%%%%%%%%
\section{Background}

%%%%%%%%%%%%%%%%%%%%%%%
\subsection{Notation}

Let $X$ be a normal variety defined over $k = \mathbf C$.
$X$ is said to be {\em $\mathbf Q$-Gorenstein} if its canonical divisor $K_X$ is a $\mathbf Q$-Cartier divisor.
For a Weil divisor $D$ on a normal variety $X$, a reflexive sheaf $\mathcal O_X(D)$ of rank $1$ is defined, and 
is also called a {\em divisorial sheaf}.

A coherent sheaf $F$ on a Cohen-Macaulay variety
$X$ is said to be {\em locally free or dual free} if it is isomorphic locally at each point to 
either $\mathcal O_X^{\oplus r}$ or $\omega_X^{\oplus r}$ for an integer $r$, where 
$\omega_X = \mathcal O_X(K_X)$ is the canonical sheaf.
We allow a locally free or dual free sheaf to be locally free at some points and locally dual free at other points.
A divisorial sheaf $\mathcal O_X(D)$ is called {\em invertible or dual invertible} it is locally free or dual free.
It is equivalent to saying that $D$ or $K_X - D$ is a Cartier divisor at each point.
Here \lq\lq dual'' means Serre-Grothendieck dual. 

A {\em quotient singularity of type $\frac 1r(a_1,\dots,a_n)$} is a singularity which is analytically isomorphic 
to the one at $0$
of the quotient space $\mathbf C^n/\mathbf Z_r$ by the action 
\[
(x_1,\dots,x_n) \mapsto (\zeta^{a_1} x_1, \dots, \zeta^{a_n} x_n),
\] 
where $\zeta$ is a primitive $r$-th root of $1$.

%%%%%%%%%%%%%%%%%%%%%%%
\subsection{Derived categories}

Let $X$ be an algebraic variety over $k$ and let $R$ be an associative $k$-algebra.
We denote by $D^b(X) = D^b(\text{coh }X)$ (resp. $D(X) = D(\text{Qcoh }X)$) 
the bounded derived category of coherent sheaves (resp. unbounded derived category of quasi-coherent sheaves) on $X$.
We also denote by $D^b(R) = D^b(\text{mod }R)$ (resp. $D(R) = D(\text{Mod }R)$)
the bounded derived category of finitely generated right $R$-modules 
(resp. unbounded derived category of right $R$-modules).

Let $\mathcal T$ be a $k$-linear triangulated category, and let 
$S$ be a set consisting of some objects of $\mathcal T$.
We denote by ${}^{\perp}S$ and $S^{\perp}$ the {\em left and right orthogonal complements} defined by
\[
\begin{split}
&{}^{\perp}S = \{t \in \mathcal T \mid \text{Hom}(t,s[p]) = 0, \forall s \in S, \forall p \in \mathbf Z\}, \\
&S^{\perp} = \{t \in \mathcal T \mid \text{Hom}(s[p],t) = 0, \forall s \in S, \forall p \in \mathbf Z\}.
\end{split}
\]
They are triangulated full subcategories of $\mathcal T$.
We denote by $\langle S \rangle$ the smallest triangulated full subcategory of $\mathcal T$ containing $S$.
In this case, $\langle S \rangle$ is said to be {\em classically generated} by $S$.
We have ${}^{\perp}S = {}^{\perp}\langle S \rangle$ and $S^{\perp} = \langle S \rangle^{\perp}$.
Furthermore, we define 
\[
\overline{\langle S \rangle} = {}^{\perp}(S^{\perp}).
\]
In particular, we have $\mathcal T = \overline{\langle S \rangle}$ if and only if $S^{\perp} \cong 0$, i.e., 
$\text{Hom}(s,t[p]) = 0$ for all $s \in S$ and all $p \in \mathbf Z$ implies that $t \cong 0$.
In this case, we say that $\mathcal T$ is {\em weakly generated} by $S$ in this paper. 
We write $\overline{\langle S \rangle}^{\mathcal T}$ for $\overline{\langle S \rangle}$ if we need to 
specify where the closure is taken. 

\begin{Lem}
$\overline{\langle S \rangle}$ is weakly generated by $S$ in the above sense.
\end{Lem}

\begin{proof}
Let $t \in \overline{\langle S \rangle}$ be an object such that $t \in S^{\perp}$.
Then $\text{Hom}(t,t) \cong 0$, hence $t \cong 0$.
\end{proof}

Let $\mathcal T$ be a $k$-linear triangulated category.
$\mathcal T$ is said to have a {\em semi-orthogonal decomposition} to triangulated full subcategories 
$\mathcal A$ and $\mathcal B$, denoted as $\mathcal T = \langle \mathcal A, \mathcal B \rangle$, if the following hold (\cite{B}): 

(1) $\text{Hom}_{\mathcal T}(b,a) = 0$ for all $a \in \mathcal A$ and $b \in \mathcal B$.

(2) $\mathcal T$ coincides with the smallest triangulated subcategory containing $\mathcal A$ and $\mathcal B$. 

We write 
\[
\mathbb R\text{Hom}_{\mathcal T}(A,B) := \bigoplus_p \text{Hom}_{\mathcal T}(A,B[p])[-p]
\]
for $A,B \in \mathcal T$.
An object $A \in \mathcal T$ is said to be an {\em exceptional object} if $\mathbb R\text{Hom}_{\mathcal T}(A,A) \cong k$.
A sequence of exceptional objects $(A_1,\dots,A_m)$ is called an {\em exceptional collection} if 
the semi-orthogonality condition $\mathbb R\text{Hom}(A_i,A_j) \cong 0$ holds for $i > j$.
It is said to be {\em full} if the $A_i$ classically generate $\mathcal T$. 

We do not require an exceptional object in $D^b(X)$ to be a perfect complex.
Therefore we do not have necessarily a semi-orthogonal decomposition of the bounded derived category arising from an 
exceptional object $A$.
Indeed we do not have finite dimensionality of $\mathbb RHom(A,\bullet)$ or $\mathbb RHom(\bullet,A)$ in general.

%%%%%%%%%%%%%%%%%%%%%%%
\subsection{$\mathbf Q$-Gorenstein smoothing}

Terminal singularities and $\mathbf Q$-Gorenstein smoothings appear naturally in the minimal model program.
Morrison-Stevens \cite{MS} classified all $3$-dimensional terminal quotient singularities.
They are singularities of types $\frac 1r(1,-1,a)$ for integers $r,a$ such that $0 < a < r$ and $(r,a) = 1$.

Let $V$ be a quotient singularity of type $\frac 1r(1,-1,a)$, and let $X = \{(x,y,z) \in V; xy = z^{sr}\}$ be a Cartier divisor, 
where $x,y,z$ are semi-invariant coordinates on $V$ and $s$ is a positive integer. 
Then $X \cong \frac 1{sr^2}(1,asr-1)$ is again a quotient singularity with semi-invariant coordinates 
$u,v$ by the rule
$x = u^{sr}$, $y = v^{sr}$, $z = uv$.

As an application of the theory of minimal models, 
Koll\'ar-Shepherd-Barron \cite{KSB} considered {\em $\mathbf Q$-Gorenstein smoothing}, a flat deformation of 
a singularity such that the canonical divisor of the total space is a $\mathbf Q$-Cartier divisor.
The above $X$ has a $\mathbf Q$-Gorenstein smoothing defined by  
\[
X_t = \{xy = z^{sr} + t\} \subset \frac 1r(1,-1,a)
\]
(cf. \cite{KSB}, Proposition 3.10).
Wahl \cite{W} already earlier considered this kind of deformation in the case $s = 1$, because the  
Milnor fiber of the deformation is a rational homology ball in this case.
In general the Milnor number of this deformation is equal to $s-1$.

The semi-universal $\mathbf Q$-Gorenstein deformation of $X$ is described as 
\[
X_{t_0,\dots,t_{s-1}} = \{xy = z^{sr} + \sum_{i=0}^{s-1} t_iz^{ir} \} \subset \frac 1r(1,-1,a)
\]
if $r > 1$, because it is lifted to a deformation of the canonical cover (cf. \cite{KMM}), which is a hypersurface singularity. 

%%%%%%%%%%%%%%%%%%%%%%%
\subsection{Exceptional vector bundles on $\mathbf P^2$ and Del Pezzo surfaces}

\cite{DLP} classified all exceptional sheaves on $\mathbf P^2$, which are automatically locally free.
\cite{GR} and \cite{R} classified all full exceptional collections of vector bundles on $\mathbf P^2$:

\begin{Thm}[\cite{R} Theorem 3.2] 
Let $(A,B,C)$ be a full exceptional collection of vector bundles on $\mathbf P^2$.
Then $(a,b,c) = \text{rank }(A,B,C)$ satisfies a Markov equation
\[
a^2+b^2+c^2 = 3abc.
\]
Moreover any triple $(A,B,C)$ is obtained from the initial triple $(\mathcal{O}(-2),\mathcal{O}(-1),\mathcal{O})$
by left and right mutations $(A,B,C) \mapsto (A,C',B)$, $(B,A',C)$ defined below
\[
\begin{split}
&0 \to C' \to \text{Hom}(B,C) \otimes B \to C \to 0, \\
&0 \to A \to \text{Hom}(A,B)^* \otimes B \to A' \to 0.
\end{split}
\]
up to cyclic permutations $(A,B,C) \mapsto (C(-3),A,B)$, twisting by line bundles $(A(m),B(m),C(m))$, 
and taking duals $(C^*,B^*,A^*)$.
\end{Thm}

We note that these exceptional collections are automatically {\em strong} in the sense that there are no higher Hom's between them.
The same holds for the Del Pezzo surfaces explained below. 

For example, we have  
\[
0 \to \mathcal{O}(-2) \to \mathcal{O}(-1)^{\oplus 3} \to \Omega^1(1) \to 0.
\]
All solutions of the Markov equation $a^2+b^2+c^2 = 3abc$ are obtained from the initial solution $(1,1,1)$
by left and right mutations up to permutations:
\[
(a,b,c) \mapsto \begin{cases} &(a,c',b), \quad c' = 3ab-c, \\
&(b,a',c), \quad a' = 3bc - a. \end{cases}
\]
They form a trivalent tree: 
\[
(1,1,1) \to (1,2,1) \to (1,5,2) \to (1,13,5), (5,29,2) \to \dots.
\]

\vskip 1pc

\cite{KN} generalized the above to full exceptional collections of vector bundles on smooth Del Pezzo surface $X$.
A {\em $3$-block exceptional collection} is an exceptional collection of vector bundles
$(A_1,\dots,A_{\alpha}; B_1, \dots, B_{\beta}; C_1, \dots, C_{\gamma})$
where members of the same block are mutually orthogonal and have the same rank: 
\[
\mathbb R\text{Hom}(A_i,A_{i'}) \cong \mathbb R\text{Hom}(B_j,B_{j'}) \cong \mathbb R\text{Hom}(C_k,C_{k'}) \cong 0
\]
and $\text{rank}(A_i,B_j,C_k) = (a,b,c)$ for any $i,j,k$. 
The triple $(a,b,c)$ satisfies a Markov equation 
\[
\alpha a^2 + \beta b^2 + \gamma c^2 = \lambda abc, \quad \lambda = \sqrt{K_X^2\alpha \beta \gamma}
\]
where $\alpha, \beta, \gamma, \lambda$ are positive integers depending on $X$ (see the table in \cite{KN} 3.5).

The left and right mutations of blocks 
\[
\begin{split}
&(A_1,\dots,A_{\alpha}; B_1, \dots, B_{\beta}; C_1, \dots, C_{\gamma}) 
\mapsto (A_1,\dots,A_{\alpha}; C'_1, \dots, C'_{\gamma}; B_1, \dots, B_{\beta}), \\ 
&(A_1,\dots,A_{\alpha}; B_1, \dots, B_{\beta}; C_1, \dots, C_{\gamma}) 
\mapsto (B_1,\dots,B_{\beta}; A'_1, \dots, A'_{\alpha}; C_1, \dots, C_{\gamma})
\end{split}
\]
are defined as follows:
\[
\begin{split}
&0 \to C'_k \to \bigoplus_j Hom(B_j,C_k) \otimes B_j \to C_k \to 0, \\
&0 \to A_i \to \bigoplus_j Hom(A_i,B_j)^* \otimes B_j \to A'_i \to 0.
\end{split}
\]
The ranks of the mutated bundles are given by
\[
c' = \frac {\lambda}{\gamma}ab - c, \quad a' = \frac {\lambda}{\alpha}bc - a.
\]

%%%%%%%%%%%%%%%%%%%%%%%
\subsection{$\mathbf Q$-Gorenstein smoothing and exceptional vector bundles}

Hacking \cite{H} proved that an exceptional vector bundle appears on a $\mathbf Q$-Gorenstein smoothing 
when $s = 1$:

\begin{Thm}[\cite{H} Theorem 1.1]
Let $X$ be a normal projective surface with a unique quotient singularity $P \in X$ of Wahl type 
$\frac 1{r^2}(1, ar-1)$. 
Let $f: \mathcal X \to \Delta$ be a one parameter flat 
deformation of $X = f^{-1}(0)$ such that the fibers $Y = f^{-1}(t)$ for $t \ne 0$ are smooth and the canonical
divisor $K_{\mathcal X}$ of the total space is a $\mathbf Q$-Cartier divisor.
Assume that $H_1(Y,\mathbf Z)$ is finite of order coprime to $r$, and that $H^2(Y, \mathcal O_Y) = 0$. 
Then, after a base change $\Delta' \to \Delta$ of degree $a$ and a shrinking of $\Delta'$ to a smaller disk, 
there exists a reflexive sheaf $\mathcal E$ on $\mathcal X' := \mathcal X \times_T T'$
such that

(a) $E_Y := \mathcal E \otimes \mathcal O_Y$ is an exceptional vector bundle of rank $r$ on $Y$, and is slope stable.

(b) $E_X := \mathcal E \otimes \mathcal O_X$ is a torsion-free sheaf on X such that its reflexive
hull $E_X^{**}$ is isomorphic to the direct sum of $r$ copies of a reflexive rank $1$ sheaf $A$, 
and the quotient $E_X^{**}/E_X$ is a torsion sheaf supported at $P \in X$.
\end{Thm}

Hacking-Prokhorov \cite{HaP} and \cite{H} classified surfaces which have smoothings to $\mathbf P^2$:

\begin{Thm}[\cite{HaP} Corollary 1.2, \cite{H} Proposition 6.2]
Let $X$ be a normal projective surface with only quotient singularities.
Assume that $X$ has a smoothing to $\mathbf P^2$.
Then $X$ is isomorphic to a weighted projective plane $\mathbf P(a^2,b^2,c^2)$ 
or its $\mathbf Q$-Gorenstein partial smoothing, where 
$a,b,c$ are positive mutually coprime integers satisfying the Markov equation 
\[
a^2+b^2+c^2 = 3abc.
\]
Moreover, $X$ is uniquely determined by its singularities up to isomorphism.
\end{Thm}

They also classified all del Pezzo surfaces with only quotient singularities such that $\rho(X) = 1$ and 
admit $\mathbf Q$-Gorenstein smoothings (\cite{HaP} Theorem 1.1).

%%%%%%%%%%%%%%%%%%%%%%%
\subsection{Pretilting objects}

Let $X$ be a projective variety over $k$.
An object $T \in D^b(X)$ is said to be {\em pretilting} if 
\[
\text{Hom}_X(T,T[p]) \cong 0
\]
for all $p \ne 0$.
Let $R_T = \text{End}(T)$ be the endomorphism ring.
It is a finite dimensional associative algebra over $k$.
$T$ is said to be {\em tilting} when it is a perfect complex and weakly generates the whole category $D^b(X)$
(in the sense that $\overline{\langle T \rangle} = D^b(X)$).
We do not require that a pretilting object $T \in D^b(X)$ is a perfect complex.
This is because we consider free or dual free sheaves in this paper which are not necessarily perfect.
For example, $\omega_X$ is not necessarily a perfect complex.
Thus $\mathbb R\text{Hom}(T,A)$ is not necessarily bounded for $A \in D^b(X)$ in general.

\begin{Lem}
Let $T$ be a pretilting object.
Define $\Phi: D(X) \to D(R_T)$ by $\Phi(\bullet) = \mathbb R\text{Hom}_X(T,\bullet)$ and
$\Psi: D(R_T) \to D(X)$ by $\Psi(\bullet) = \bullet \otimes^{\mathbb L}_{R_T} T$.
Then $\text{Im}(\Psi) = \overline{\langle T \rangle}$ and $\Psi$ induces
an equivalence $D(R_T) \cong \overline{\langle T \rangle}$.
\end{Lem}

\begin{proof}
$\Psi$ is a left adjoint of $\Phi$; $\text{Hom}_{D(X)}(\Psi(A), B) \cong \text{Hom}_{D(R_T)}(A, \Phi(B))$ for
$A \in D(R_T)$ and $B \in D(X)$.
Moreover the adjunction morphism of functors $\text{Id}_{D(R_T)} \to \Phi \Psi$ is an equivalence.
Indeed we have $\Phi \Psi(A) = \mathbb R\text{Hom}_X(T, A \otimes^{\mathbb L}_{R_T} T) \cong A$ for 
$A \in D(R_T)$.
We have also $\text{Hom}(A, B) \cong \text{Hom}(A, \Phi\Psi(B)) \cong \text{Hom}(\Psi(A), \Psi(B))$.
Therefore $\Psi$ induces an equivalence $D(R_T) \cong \text{Im}(\Psi)$.

Let $A \in D(R_T)$ and $B \in T^{\perp}$.  
Then we have $\mathbb R\text{Hom}_{D(X)}(\Psi(A), B) \cong 
\mathbb R\text{Hom}_{D(R_T)}(A, \Phi(B)) \cong 0$, hence $\text{Im}(\Psi) \subset \overline{\langle T \rangle}$.

Conversely, let $C \in \overline{\langle T \rangle}$.
There is a distinguished triangle
\[
\Psi\Phi(C) \to C \to C' \to \Psi\Phi(C)[1]
\]
for some $C' \in D(X)$.
Since $\Psi\Phi(C) \in \overline{\langle T \rangle}$, we have $C' \in \overline{\langle T \rangle}$.
Since $\Phi\Psi\Phi(C) \cong \Phi(C)$, we have $\Phi(C') \cong 0$, i.e., $C' \in T^{\perp}$.
Then we have $\text{Hom}_{D(X)}(C', C') \cong 0$, hence $C' \cong 0$.
Therefore $\overline{\langle T \rangle} \subset \text{Im}(\Psi)$.
\end{proof}

We consider a special case where a pretilting object $T$ is a versal NC deformation of a simple coherent sheaf $S$ on $X$ 
in this paper (cf. \cite{NC multi}):

\begin{Lem}
Assume that a pretilting object $T$ is a versal NC deformation of a simple coherent sheaf $S$ on $X$.
Then $\Psi$ induces induces an equivalence $\Psi^b: D^b(R_T) \cong \langle S \rangle \subset D^b(X)$.
Moreover $\langle S \rangle = \overline{\langle T \rangle} \cap D^b(X)$, where the closure is
taken in $D(X)$.
\end{Lem}

\begin{proof}
$T$ is flat over $R_T$ and $k \otimes_{R_T} T \cong S$ by the definition of an NC deformation and
\cite{NC multi} Lemmas 4.4 and 4.5 with Corollary 4.6.
It follows that the functor $\Psi^b: \text{mod }R_T \to \text{Coh }X$ defined by 
$\Psi^b(\bullet) = \bullet \otimes_{R_T} T$ is exact, and induces a triangulated functor 
\[
\Psi^b: D^b(R_T) \to D^b(X).
\]
Since $R_T$ is finite dimensional, $D^b(R_T)$ is classically generated by $k$; $D^b(R_T) = \langle k \rangle$.
We have $\Phi(k) \cong S$, hence $\text{Im}(\Psi^b) = \langle S \rangle$.
The equivalence $\Phi: D(R_T) \cong \overline{\langle T \rangle}^{D(X)}$ induces 
an equivalence $\Phi^b: D^b(R_T) \cong \langle S \rangle$.

Let us take a $A \in \overline{\langle T \rangle} \cap D^b(X)$.
We have a distinguished triangle arising from a natural morphism 
\[
\Psi\Phi(A) \to A \to B \to \Psi\Phi(A)[1].
\]
Since $\Phi\Psi\Phi(A) \cong \Phi(A)$, we have $\Phi(B) \cong 0$, i.e., $B \in T^{\perp}$.
On the other hand, $\Psi\Phi(A)$ and $A$ belong to $\overline{\langle T \rangle}$, hence so does $B$.
Therefore $B \cong 0$, hence $A \cong \Psi\Phi(A)$.
We note that, if $\Phi(A)$ is not bounded, then so is $\Psi\Phi(A)$, a contradiction.
\end{proof}

%%%%%%%%%%%%%%%%%%%%%%%
\subsection{Miscellaneous}

\begin{Lem}\label{local vanishing}
Let $F$ and $G$ be coherent sheaves on a Cohen-Macaulay variety $X$. 
Assume that either one of the following holds at each point:
(1) $F$ is locally free, or (2) $F$ is maximally Cohen-Macaulay and $G$ is locally dual free. 
Then all higher local extensions vanish: $\mathcal{E}xt^i(F,G) = 0$ for $i > 0$.
\end{Lem}

\begin{proof}
The case (1) is clear.
For the case (2), we may assume that $G = \omega_X$.
By the local duality theorem (\cite{Hartshorne} Theorem V.6.2), 
we have 
\[
\text{Hom}(\mathcal{E}xt^i(F, G),I) \cong \mathcal H^{n-i}_x(F) \cong 0
\]
for $i > 0$, where $n = \dim X$, $x \in X$ is a closed point and $I$ is the injective hull of $k(x)$. 
Hence $\mathcal{E}xt^i(F, G) \cong 0$.
\end{proof}

We will need the following proposition:

\begin{Prop}\label{extension}
Let $f: \mathcal X \to \Delta$ be a flat projective morphism from a Cohen-Macaulay variety to a disk, and 
let $F$ be a locally free or dual free sheaf on $X = f^{-1}(0)$.
Assume that $F$ is pretilting. 
Then, after shrinking $\Delta$ if necessary, there exists a coherent sheaf $\mathcal F$ on $\mathcal X$ uniquely up to isomorphisms
which satisfies the following:

(1) $\mathcal F$ is locally free or dual free.

(2) $\mathcal F \otimes \mathcal O_X \cong F$.

(3) $\mathcal F \otimes \mathcal O_Y$ is pretilting for $Y = f^{-1}(t)$ with $t \ne 0$. 

(3) $\dim \text{End}(F) = \dim \text{End}(\mathcal F \otimes \mathcal O_Y)$.
\end{Prop}

\begin{proof} 
We take an open covering $X = \bigcup U_i$ such that $F \vert_{U_i}$ is isomorphic to 
either $\mathcal O_{U_i}^{\oplus r}$
or $\omega_{U_i}^{\oplus r}$.
There are gluing isomorphisms $\phi_{ij}: F_j \vert_{U_{ij}} \to F_i \vert_{U_{ij}}$ for $U_{ij} = U_i \cap U_j$.

Let $X_m$ be the $m$-th infinitesimal neighborhood of $X$ in $\mathcal X$ defined by $t^{m+1}=0$, where $t$ is the 
parameter on $\Delta$.
We have $X_0 = X$.

We extend $F$ to a locally free or dual free sheaf $F^{(m)}$ on $X_m$ by induction on $m$.
Suppose that we have already $F^{(m)}$.
We have $F^{(m)}_i = F^{(m)} \vert_{U_i} \cong \mathcal O_{U^{(m)}_i}^{\oplus r}$ or $\omega_{U^{(m)}_i}^{\oplus r}$, 
where $U^{(m)}_i = X_m \vert_{U_i}$.
There are gluing isomorphisms $\phi^{(m)}_{ij}: F^{(m)}_j \vert_{U_{ij}} \to F^{(m)}_i \vert_{U_{ij}}$.

Let $F^{(m+1)}_i = \mathcal O_{U^{(m+1)}_i}^{\oplus r}$ or $\omega_{U^{(m+1)}_i}^{\oplus r}$, 
and let $\phi^{(m+1)}_{ij}: F^{(m+1)}_j \vert_{U_{ij}} \to F^{(m+1)}_i \vert_{U_{ij}}$ 
be any isomorphisms extending the $\phi^{(m)}_{ij}$.
If they satisfy the cocycle condition 
$\phi^{(m+1)}_{ij}\phi^{(m+1)}_{jk} = \phi^{(m+1)}_{ik}: F^{(m+1)}_k \vert_{U_{ijk}} \to F^{(m+1)}_i \vert_{U_{ijk}}$ for 
$U_{ijk} = U_i \cap U_j \cap U_k$, then we obtain $F^{(m+1)}$.
Therefore we define
\[
\delta_{ijk} = \phi^{(m+1)}_{ij}\phi^{(m+1)}_{jk} - \phi^{(m+1)}_{ik} 
\in t^{m+1}\mathcal Hom_{U_{ijk}}(F_k \vert_{U_{ijk}}, F_i \vert_{U_{ijk}}).
\]
$\{\delta_{ijk}\}$ is a closed $2$-cochain in the following sense:
\[
\begin{split}
&\phi_{ij} \delta_{jkl} - \delta_{ikl} + \delta_{ijl} - \delta_{ijk} \phi_{kl}\\
&= \phi^{(m+1)}_{ij} (\phi^{(m+1)}_{jk}\phi^{(m+1)}_{kl} - \phi^{(m+1)}_{jl}) - (\phi^{(m+1)}_{ik}\phi^{(m+1)}_{kl} - \phi^{(m+1)}_{il}) \\
&+ (\phi^{(m+1)}_{ij}\phi^{(m+1)}_{jl} - \phi^{(m+1)}_{il}) - (\phi^{(m+1)}_{ij}\phi^{(m+1)}_{jk} - \phi^{(m+1)}_{ik}) \phi^{(m+1)}_{kl} \\
&= 0
\end{split}
\]
where we used $t^{m+1}\phi^{(m+1)}_{ij} = t^{m+1}\phi_{ij}$.

Since $H^2(X, \mathcal End(F)) \cong 0$, there is a $1$-cochain $\{\epsilon_{ij}\}$ such that 
$\epsilon_{ij} \in t^{m+1} \mathcal Hom_{U_{ij}}(F_j \vert_{U_{ij}}, F_i \vert_{U_{ij}})$ and 
\[
\delta_{ijk} = \phi_{ij} \epsilon_{jk}  - \epsilon_{ik} + \epsilon_{ij} \phi_{jk}.
\]
More precisely, we have
\[
\delta_{ijk} \phi_{ki} = \phi_{ij} \epsilon_{jk} \phi_{ki} - \epsilon_{ik} \phi_{ki} + \epsilon_{ij} \phi_{ji}
\]
where we used the cocycle condition for the $\phi_{ij}$.
Then it follows 
\[
(\phi^{(m+1)}_{ij} - \epsilon_{ij}) (\phi^{(m+1)}_{jk} - \epsilon_{jk}) = \phi^{(m+1)}_{ik} - \epsilon_{ik}
\]
and we obtain new gluing isomorphisms for the $F^{(m+1)}_i$ satisfying the cocycle condition.
Hence we obtain $F^{(m+1)}$.

Thus we have a formal deformation of $F$ on $\mathcal X$ relatively over the disk $\Delta$.
Since there exists a moduli space of $F$ on $\mathcal X/\Delta$ (the Quot scheme), 
the deformation extends to $\mathcal X$ to yield $\mathcal F$ if we shrink $\Delta$.   
Since $H^i(X, \mathcal{E}nd(F)) \cong Ext^i(F,F) \cong 0$ for $i > 0$, 
it follows that $H^i(Y, \mathcal{E}nd(\mathcal F) \otimes \mathcal O_Y) \cong 0$ for $Y = f^{-1}(t)$ and $i > 0$
by the upper-semi-continuity theorem.
Therefore $\dim \text{End}(\mathcal F \otimes \mathcal O_{f^{-1}(t)})$ is constant and 
$Ext^i(\mathcal F \otimes \mathcal O_Y,\mathcal F \otimes \mathcal O_Y) = 0$ for $i > 0$.
\end{proof}

%%%%%%%%%%%%%%%%%%%%%%%
\subsection{Motivating example $\mathbf P(1,1,4)$}

We consider $X = \mathbf P(1,1,d)$, the cone over a normal rational curve of degree $d$ (\cite{NC multi} Example 5.7).
We will calculate this example as a particular case of Theorem \ref{Pa1a2a3}.

Let $l$ be a generator of the cone, and let 
$\mathcal{O}_X(1) = \mathcal{O}_X(l)$ be the corresponding divisorial sheaf.
We have $\mathcal{O}_X(K_X) \cong \mathcal{O}_X(-d-2)$.

We construct a sheaf $F$ (denoted by $G$ in \cite{NC multi}) by a universal extension
\[
0 \to \mathcal{O}_X(-1)^{\oplus d-1} \to F \to \mathcal{O}_X(-1) \to 0.
\]
Then $F$ is a locally free sheaf of rank $d$.
$F$ is pretilting, and we have 
\[
R_F = \text{End}(F) \cong k[[x_1,\dots,x_{d-1}]]/(x_1,\dots,x_{d-1})^2.
\]
$F$ is a versal non-commutative deformation of $\mathcal{O}_X(-1)$ over $R_F$.

We have semi-orthogonal decompositions:
\[
D^b(X) = \langle \mathcal{O}_X(-2), \mathcal{O}_X(-1), \mathcal{O}_X \rangle 
=  \langle \mathcal{O}_X(-1), \mathcal{O}_X, \mathcal{O}_X(d) \rangle.
\]
We note that $\mathcal O_X(-2)$ is dual invertible at the vertex of the cone, and is not a perfect complex if $d > 2$.
We can also write
\[
D^b(X) = \langle \mathcal{O}_X(-2), \overline F, \mathcal{O}_X \rangle 
=  \langle \overline F, \mathcal{O}_X, \mathcal{O}_X(d) \rangle
\]
where we abbreviate $\overline F = \overline{\langle F \rangle} \cap D^b(X)$.

Here we would like to correct an error in \cite{NC multi} Example 5.7.
We claimed that there is a semi-orthogonal decomposition
$D^b(X) = \langle \mathcal O_X(-d),F,\mathcal O_X \rangle$, but it is false.
Indeed we have $\mathbb R\text{Hom}(F,\mathcal O_X(-d)) \not\cong 0$ by the following calculation.
We have an exact sequence
\[
0 \to F \to  \mathcal O_X^{\oplus d} \to \mathcal O_C(d-1) \to 0
\]
where $C \in \vert \mathcal O_X(d) \vert$ is a curve at infinity.
Since $RH(X,\mathcal O_X(-d)) \cong 0$, we have
\[
\begin{split}
&\mathbb R\text{Hom}(F,\mathcal O_X(-d)) \cong \mathbb R\text{Hom}(\mathcal O_C(d-1),\mathcal O_X(-d))[1] \\
&\cong \mathbb R\text{Hom}(\mathcal O_X(-d),\mathcal O_C(-3))^*[-1]
\cong \mathbb R\Gamma(C,\mathcal O_C(d-3))^*[-1] \not\cong 0
\end{split}
\]
for $d > 2$.

In the case $d = 4$, $X$ has a $\mathbf Q$-Gorenstein smoothing to $\mathbf P^2$.
Let $V = \mathbf P(1,1,1,2)$ be the projective cone over a Veronese surface, and let $x,y,z,t$
be the semi-invariant coordinates.
It has a terminal quotient singularity at the vertex.
We can embed $X$ in $V$ by an equation $xy = z^2$, then a smoothing is given by
a linear system
\[
\mathcal X = \{xy = z^2 + st\} \subset V \times \mathbf P^1
\]
where $s$ is an inhomogeneous coordinate on $\mathbf P^1$.
$s = \infty$ corresponds to the plane at infinity.
Fibers except $X$ are isomorphic to $\mathbf P^2$, and
$2K_{\mathcal X}$ is invertible.

The fibers of the $\mathbf Q$-Gorenstein smoothing have the following semi-orthogonal decompositions:
\[
\begin{split}
&D^b(X) = \langle \mathcal{O}_X(-2), \mathcal{O}_X(-1), \mathcal{O}_X \rangle \\
&D^b(\mathbf P^2) = \langle \mathcal{O}_{\mathbf P^2}(-1), \Omega^1_{\mathbf P^2}(1), 
\mathcal{O}_{\mathbf P^2} \rangle.
\end{split}
\]
We have the following correspondence.
The dual invertible sheaf $\mathcal{O}_X(-2)$ deforms to $\mathcal{O}_{\mathbf P^2}(-1)$, 
because $K_X = \mathcal{O}_X(-6)$ deforms to $K_{\mathbf P^2} = \mathcal{O}_{\mathbf P^2}(-3)$.
The versal NC deformation $F$ of $\mathcal{O}_X(-1)$ deforms to $\Omega^1_{\mathbf P^2}(1)^{\oplus 2}$, 
and the endomorphism ring $R_F \cong k[[x_1,x_2,x_3]]/(x_1,x_2,x_3)^2$ deforms to  
$Mat(2,k)$.

Indeed there are exact sequences 
\[
\begin{split}
&0 \to F \to \mathcal{O}_X^{\oplus 4} \to \mathcal{O}_C(3) \to 0, \\
&0 \to \Omega^1_{\mathbf P^2}(1)^{\oplus 2} \to \mathcal{O}_{\mathbf P^2}^{\oplus 4} \to \mathcal{O}_C(3) \to 0, 
\end{split}
\]
where $C = X \cap \mathbf P^2$ is the curve at infinity.

%%%%%%%%%%%%%%%%%%%%%%%
\section{Non-commutative deformation on a surface with a cyclic quotient singularity}

%%%%%%%%%%%%%%%%%%%%%%%
\subsection{Generalities}

We recall the theory of multi-pointed non-commutative (NC) deformations.
Let $k^m$ be a direct product ring for a positive integer $m$.
We denote by $\text{Art}_m$ the category of 
associative {\em augmented} $k^m$-algebras $R$, i.e., there are ring homomorphisms $k^m \to R \to k^m$ whose 
composition is the identity, such that 
$R$ is finite dimensional as a $k$-vector space and that the two-sided ideal
$\frak m = \text{Ker}(R \to k^m)$ is nilpotent.

Let $F$ be an object in a $k$-linear abelian category such as the category of coherent sheaves $\text{Coh}(X)$.
$F$ has a left $k^m$-module structure if and only if it has a form of a direct sum $F = \oplus_{i=1}^m F_i$.
An {\em $m$-pointed non-commutative (NC) deformation} $\tilde F$ of $F$ over $R \in \text{Art}_m$ 
is a flat left $R$-module object in the abelian category together with a fixed isomorphism 
\[
F \to k^m \otimes_R \tilde F.
\]
The infinitesimal deformation theory in the non-commutative setting is very similar to the commutative one, 
where the base ring $R$ is assumed to be commutative.  
In particular, there exists a {\em semi-universal} or {\em versal} deformation over a pro-object 
$\hat R \in \text{Art}\hat{}_m$ which is uniquely determined up to an isomorphism, where
$\text{Art}\hat{}_m$ is the category of augmented $k^m$-algebras $R$ with $\frak m = \text{Ker}(R \to k^m)$
such that 
$R/\frak m^n \in \text{Art}_m$ for all $n > 0$.

By taking the functor $\otimes \tilde F$ to an extension of algebras
\[
0 \to k_i \to R' \to R \to 0
\]
there is an extension of deformations
\[
0 \to F_i \to \tilde F' \to \tilde F \to 0
\]
where $k_i$ is an ideal which is isomorphic to $k$ 
and is annihilated by all components of $k^m$ except the $i$-th component.
Therefore the deformations of $F$ are obtained by {\em extensions} of the $F_i$.

If $\text{End}(F) \cong k^m$, then $F$ is said to be a {\em simple collection} and the theory is particularly 
simple (\cite{NC multi} Lemmas 4.4 and 4.5 with Corollary 4.6).
In particular we have $\text{End}(\tilde F) \cong R$ if $\tilde F$ is obtained by a succession of 
non-trivial extensions of the $F_i$, e.g., the versal deformation.

For example, if $F = \mathcal{O}_X(-1)$ on $X = \mathbf P(1,1,4)$ with $m = 1$, then the 
versal deformation $\tilde F$ is obtained as a universal extension 
\[
0 \to \mathcal{O}_X(-1)^3 \to \tilde F \to \mathcal{O}_X(-1) \to 0.
\]

%%%%%%%%%%%%%%%%%%%%%%%
\subsection{$2$-dimensional cyclic quotient singularity}

The versal deformation of a divisorial sheaf on a surface with a cyclic quotient singularity is determined by
Karmazyn-Kuznetsov-Shinder \cite{KKS} (Lemma 3.13, Theorem 3.16, Proposition 6.7):

\begin{Thm}\label{KKS}
Let $X = \frac 1r(1,a)$ be a quotient singularity of dimension $2$, where $0 < a < r$ and $(r,a) = 1$, and let 
$C \subset X$ be the image of the coordinate divisor corresponding to the weight $1$.
Let $r/(r-a) = [c_1,\dots,c_l]$ be an expansion to continued fractions.
Then the versal NC deformation $\tilde F$ of a divisorial sheaf $F = \mathcal{O}_X(-1) = \mathcal{O}_X(- C)$ 
is a locally free sheaf of rank $r$ on $X$, and the 
parameter algebra of the deformation is isomorphic to a {\em Kalck-Karmazyn algebra} 
$R = k\langle z_1,\dots,z_l \rangle/I$, 
where $I$ is a two-sided ideal in a non-commutative polynomial algebra generated by
\[
\begin{split}
&z_j^{c_j}, \quad \forall j, \\ 
&z_jz_k, \quad j < k, \\ 
&z_j^{c_j-1}z_{j-1}^{c_{j-1}-2} \dots z_{k+1}^{c_{k+1}-2}z_k^{c_k-1}, \quad k < j.
\end{split}
\]
\end{Thm}

We note that there is another expansion $r/a = [d_1, \dots, d_m]$ to continued fractions 
corresponding to the Hirzebruch-Jung string
of exceptional divisors on the minimal resolution of $X$.

We note also that we consider a divisorial sheaf $F = \mathcal{O}_X(-1)$ corresponding to a divisor with 
negative coefficient instead of the positive one $\mathcal{O}_X(a)$
as in \cite{KKS} in order to simplify the notation.

$F$ is a simple collection with one element on a suitable compactification of $X$.
In the following, we construct
the versal NC deformation $\tilde F$ of $F = \mathcal{O}_X(-1)$ using an argument from \cite{KKS}, and
deduce its local freeness by more direct elementary method.

\vskip 1pc

Let $X \subset X'$ be a compactification to a normal projective surface 
which has only one singularity and 
vanishing cohomologies $H^i(X',\mathcal O_{X'}) \cong 0$ for $i > 0$
e.g., a rational surface, let $C'$ be the closure of $C$, and let $F' =  \mathcal{O}_{X'}(-C')$ be 
the corresponding divisorial sheaf.
Then we have $\mathcal End(F') \cong \mathcal O_{X'}$, 
and the extension sheaves $\mathcal Ext^i(F',F')$ for $i > 0$
are skyscraper sheaves supported at the singular point. 
Then we have $\text{Ext}^i_{X'}(F',F') \cong H^0(X',\mathcal Ext^i(F,F))$. 
Therefore we can consider the deformations of $F$ and $F'$ to be the same, and
the process to obtain the versal NC deformation is the same on $X$ and $X'$.

More generally, even if $X'$ has other isolated singularities, we have the same argument if 
$F'$ is locally free or dual free at these points, 
because $\mathcal End(F') \cong \mathcal O_{X'}$ and $\mathcal Ext^i(F',F') \cong 0$ for $i > 0$ 
there by Lemma \ref{local vanishing}.

Let $f: Y \to X = \frac 1r(1,a)$ be the minimal resolution, 
let $E_1,\dots,E_m$ be the exceptional curves, which are isomorphic to $\mathbf P^1$'s, and let 
$C' = f_*^{-1} C$ be the strict transform, so that 
$(C', E_m,\dots,E_1)$ is a chain of curves in this order.
$\sum_{j=1}^m E_j$ is the fundamental cycle of the singularity, and we have
$\frak m_x \mathcal{O}_Y = \mathcal{O}_Y(-\sum_{j=1}^m E_j)$, where $x \in X$ is the singular point.
Let $L_m = \mathcal{O}_Y(-C')$, and define $L_{j-1} = L_m(-E_m-\dots-E_j)$ for $1 \le j \le m$.
Thus we have $f_*L_j = F$ for all $0 \le j \le m$.

First we consider $m+1$ pointed NC deformations of $L = \oplus_{j=0}^m L_j$ on $Y$, and let 
$\tilde L = \oplus_{j=0}^m \tilde L_j$ be the versal NC deformation.
We note that $L$ is not a simple collection, but the NC deformation theory is not complicated in this case, 
because the $L_i$ are exceptional objects which are semi-orthogonal, i.e., the vanishing
\[
\mathbb R\text{Hom}(L_j,L_{j'}) \cong \mathbb R\Gamma(Y, \mathcal{O}_Y(-E_j - \dots - E_{j'+1}) \cong 0
\]
holds for $j > j'$, and therefore non-trivial extensions of the $L_j$ occur only in one direction.
We note also that all sheaves appearing in the process below are locally free.

The direct summand $\tilde L_0$ is constructed in the following way.
It is the largest one among the $\tilde L_i$.
It is obtained by iterated {\em universal extensions} $G_j$ starting from $G_0 = L_0$ and defined inductively by  
\begin{equation}\label{universal}
0 \to Ext^1(G_j, L_{j+1})^* \otimes L_{j+1} \to G_{j+1} \to G_j \to 0.
\end{equation}
We note that any extension of $G_j$ by $L_{j+1}$ is obtained by a pull-back from $G_{j+1}$.
We also note that $Ext^1(G_j, L_i) = 0$ for all $i \le j$.
Therefore we have $G_m = \tilde L_0$.

Another summand $\tilde L_i$ for $i > 0$ is similarly constructed from $L_i$ by successively taking universal extensions
by $L_{i+1}, \dots, L_m$, but we do not need it.

\begin{Lem}\label{3.2}
Let $\tilde F = f_*\tilde L_0$.
Then $\tilde F$ is a locally free sheaf of rank $r$ at the singular point.
\end{Lem}

It is proved in \cite{KKS} Proposition 6.7.
We give an alternative elementary direct proof.
The construction here is \lq\lq reversed'' using divisors with negative coefficients.
It is also a good example of the multi-pointed NC deformation theory.

\begin{proof}
It is sufficient to prove that the restriction of $\tilde L_0$ to the fundamental cycle is trivial:
\[
\tilde L_0 \otimes \mathcal O_{\sum_{j=1}^m E_j} \cong \mathcal{O}_{\sum_{j=1}^m E_j}^{\oplus r}.
\]
Indeed the generating sections of the right hand side can be extended to sections of $\tilde L_0$
since $R^1f_*\mathcal O_Y(-n\sum_{j=1}^m E_j) = 0$ for all $n > 0$.

Let $r/a = [d_1,\dots,d_m]$ be an expansion to continued fractions.
Then by \cite{KKS} Lemma 3.3, we have
\[
r = \text{det} \left( \begin{matrix}
d_1 & 1 & 0 & 0 & \dots & 0 \\
1 & d_2 & 1 & 0 & \dots & 0 \\
0 & 1 & d_3 & 1 & \dots & 0 \\
\dots & \dots & \dots & \dots & \dots & \dots \\
0 & 0 & 0 & 0 & \dots & d_m
\end{matrix} \right), \quad 
a = \text{det} \left( \begin{matrix}
d_2 & 1 & 0 & 0 & \dots & 0 \\
1 & d_3 & 1 & 0 & \dots & 0 \\
0 & 1 & d_4 & 1 & \dots & 0 \\
\dots & \dots & \dots & \dots & \dots & \dots \\
0 & 0 & 0 & 0 & \dots & d_m
\end{matrix} \right).
\]
The self intersection numbers are given by $d_i = - E_i^2 \ge 2$, and 
\[
\text{deg}_{E_j}(L_i) = (a_{ij}) = 
\left( \begin{matrix}
d_1-1 & d_2-2 & d_3-2 & \dots & d_{m-2}-2 & d_{m-1}-2 & d_m-2 \\
-1 & d_2-1 & d_3-2 & \dots & d_{m-2}-2 & d_{m-1}-2 & d_m-2 \\
0 & -1 & d_3-1 & \dots & d_{m-2}-2 & d_{m-1}-2 & d_m-2 \\
\dots & \dots & \dots & \dots & \dots & \dots & \dots \\
0 & 0 & 0 & \dots & -1 & d_{m-1}-1 & d_m-2 \\
0 & 0 & 0 & \dots & 0 & -1 & d_m-1 \\
0 & 0 & 0 & \dots & 0 & 0 & -1 
\end{matrix} \right)
\]
for $0 \le i \le m$ and $1 \le j \le m$.

We shall prove by induction on $i$ ($1 \le i \le m$) that 
\[
G_i \otimes \mathcal O_{E_j} \cong \mathcal{O}_{E_j}^{\oplus r_i}
\]
for $j \le i$, where $r_i = \text{rank}(G_i)$ is given by
\[
r_i = \text{det} \left( \begin{matrix}
d_1 & 1 & 0 & 0 & \dots & 0 \\
1 & d_2 & 1 & 0 & \dots & 0 \\
0 & 1 & d_3 & 1 & \dots & 0 \\
\dots & \dots & \dots & \dots & \dots & \dots \\
0 & 0 & 0 & 0 & \dots & d_i
\end{matrix} \right)
\]
and that 
$G_i \otimes \mathcal O_{E_j}$ for $j > i$ are trivial extensions of the $L_k \otimes \mathcal O_{E_j}$ for $k \le i$.
We note that $L_k \otimes \mathcal O_{E_j} \cong \mathcal O_{\mathbf P^1}(d_j-1)$ if $j = k+1$, 
and $\cong \mathcal O_{\mathbf P^1}(d_j-2)$
if $j > k+1$.

Let $Y'$ be the compactification of $Y$ which coincides with $X'$ outside the quotient singularity.
For $i = 1$, the restriction map
\[
\text{Ext}^1_{Y'}(L_0,L_1) \to \text{Ext}^1_{E_1}(L_0 \otimes \mathcal O_{E_1},L_1 \otimes \mathcal O_{E_1})
\]
is bijective, because so is
\[
\begin{split}
&H^1(Y',L_1 \otimes L_0^*) \to H^1(E_1,L_1 \otimes L_0^* \otimes \mathcal O_{E_1}) \\
&\cong H^1(\mathbf P^1,\mathcal O_{\mathbf P^1}(-d_1)) \cong k^{d_1-1}
\end{split}
\]
since $H^k(Y',\mathcal O_{Y'}) = 0$ for $k=1,2$.
The corresponding extension on $\mathbf P^1$ is given by
\[
0 \to \mathcal O_{\mathbf P^1}(-1)^{d_1-1} \to \mathcal O_{\mathbf P^1}^{d_1} \to 
\mathcal O_{\mathbf P^1}(d_1-1) \to 0
\]
hence our assertion holds with $r_1 = d_1$.
The extensions on $E_j$ for $j \ge 2$ are trivial, because 
$L_1 \otimes L_0^* \otimes \mathcal O_{E_2} \cong \mathcal O_{\mathbf P^1}(-1)$ and 
$L_1 \otimes L_0^* \otimes \mathcal O_{E_j} \cong \mathcal O_{\mathbf P^1}$ for $j > 2$.

We assume that our assertion holds for $G_i$ and prove it for $G_{i+1}$.
Since $L_{i+1} \otimes \mathcal O_{E_j} \cong \mathcal O_{E_j}$ for $j \le i$, 
we have $G_{i+1} \otimes \mathcal O_{E_j} \cong \mathcal{O}_{E_j}^{\oplus r_{i+1}}$ for these $j$.
Since $L_{i+1} \otimes \mathcal O_{E_j} \cong \mathcal{O}_{E_j}(d_j-1)$ or $\mathcal{O}_{E_j}(d_j-2)$ for $j > i+1$,
our assertion holds for these $j$ too.
Therefore we have only to prove that 
$G_{i+1} \otimes \mathcal O_{E_{i+1}} \cong \mathcal{O}_{E_{i+1}}^{\oplus r_{i+1}}$.

We note that 
\[
G_i \otimes \mathcal O_{E_{i+1}} \cong \mathcal O_{\mathbf P^1}(d_{i+1}-2)^{r_i} \oplus 
\mathcal O_{\mathbf P^1}(d_{i+1}-1)^{r_i-r_{i-1}}
\]
by the induction hypothesis.
We also note that $\text{Ext}^1(G_i,L_i) \cong 0$,
since the extension (\ref{universal}) is universal and $L_i$ is exceptional.
The restriction map
\[
\text{Ext}^1_{Y'}(G_i,L_{i+1}) \to 
\text{Ext}^1_{E_{i+1}}(G_i \otimes \mathcal O_{E_{i+1}},L_{i+1} \otimes \mathcal O_{E_{i+1}})
\]
is bijective again, because so is
\[
H^1(Y',L_{i+1} \otimes G_i^*) \to H^1(E_{i+1},L_{i+1} \otimes G_i^* \otimes \mathcal O_{E_{i+1}})
\]
since 
\[
H^1(Y',L_{i+1}(-E_{i+1}) \otimes G_i^*) \cong \text{Ext}^1_{Y'}(G_i,L_i) \cong 0
\]
and 
\[
H^2(Y',L_{i+1}(-E_{i+1}) \otimes G_i^*) \cong H^0(Y',G_i \otimes L_i^* \otimes K_{Y'})^* \cong 0. 
\]
We have
\[
\text{Ext}^1_{E_{i+1}}(G_i \otimes \mathcal O_{E_{i+1}},L_{i+1} \otimes \mathcal O_{E_{i+1}})
\cong H^1(\mathbf P^1,\mathcal O_{\mathbf P^1}(-d_{i+1}+1)^{r_{i-1}} 
\oplus \mathcal O_{\mathbf P^1}(-d_{i+1})^{r_i-r_{i-1}})
\]
and the corresponding extension on $\mathbf P^1$ is given by
\[
\begin{split}
&0 \to \mathcal O_{\mathbf P^1}(-1)^{d_{i+1}-2} \to \mathcal O_{\mathbf P^1}^{d_{i+1}-1} \to 
\mathcal O_{\mathbf P^1}(d_{i+1}-2) \to 0, \\
&0 \to \mathcal O_{\mathbf P^1}(-1)^{d_{i+1}-1} \to \mathcal O_{\mathbf P^1}^{d_{i+1}} \to 
\mathcal O_{\mathbf P^1}(d_{i+1}-1) \to 0.
\end{split}
\]
Therefore our assertion holds with 
\[
\text{rank}(G_{i+1}) = r_{i-1}(d_{i+1}-1) + (r_i-r_{i-1})d_{i+1}
= d_{i+1}r_i - r_{i-1} = r_{i+1}.
\]
In particular, the vector bundle $\tilde L_0$ has rank $r$ and is restricted to a trivial bundle on each $E_j$.
This was to be proved.
\end{proof}

The following is an alternative proof of a result proved in \cite{KKS} Proposition 6.7 (iv):

\begin{Lem}\label{3.3}
$\tilde F = f_*\tilde L_0$ is the versal $1$-pointed NC deformation of $F = \mathcal{O}_X(-1)$.
\end{Lem}

\begin{proof}
We use \cite{NC multi} Corollary 4.11.
We consider on a compactification $X'$ of $X$ which has only one singularity and 
such that $H^i(X',\mathcal O_{X'}) = 0$ for $i > 0$.
We denote by $F'$ and $\tilde F'$ respectively the extensions of $F$ and $\tilde F$ to $X'$ 
defined by the closure of the curve $C$.
Since $\tilde F'$ is locally free, it has no more extension by $F'$.
Indeed we have 
\[
Ext^1(\tilde F', F') \cong H^0(X', \mathcal Ext^1(\tilde F', F')) \cong 0.
\]
We note here that $H^1(X', \mathcal Hom(\tilde F', F')) \cong 0$, because $\mathcal Hom(\tilde F', F')$ is an extension of 
$\mathcal Hom(F', F') \cong \mathcal O_{X'}$. 
Therefore the versality follows if $\text{Hom}_{X'}(\tilde F, F) \cong k$, which says that there is no trivial extension in the 
process to obtain $\tilde F$ (cf. \cite{NC multi}).

Since $\tilde F$ is locally free, we have $f^*\tilde F \cong \tilde L_0$, hence 
\[
\text{Hom}_{X'}(\tilde F, F) \cong \text{Hom}_{Y'}(\tilde L_0, L_0)
\]
where $Y'$ is the corresponding compactification of $Y$.
Since $\text{Hom}_{Y'}(L_j,L_0) = 0$ for $j > 0$, 
$\text{Hom}_{Y'}(G_j, L_0) \to \text{Hom}_{Y'}(G_{j+1},L_0)$ are bijective for all $j$, hence 
$\text{Hom}_{Y'}(\tilde L_0, L_0) \cong k$.
\end{proof}

%%%%%%%%%%%%%%%%%%%%%%%
\subsection{NC deformation on weighted projective plane}

Let $X = \mathbf P(a,b,c)$ be a weighted projective plane, where $(a,b,c) = (a_1,a_2,a_3)$ 
are pairwise coprime positive integers.
There are at most $3$ singular points 
$P_1 \cong \frac 1a(b,c)$, $P_2 \cong \frac 1b(a,c)$ and $P_3 \cong \frac 1c(a,b)$.
We take a positive integer $m$ such that $a \vert m$ and $m \equiv 1 \mod c$.

We define divisorial sheaves 
$L_3 = \mathcal{O}_X(-mC_{13})$, $L_2 = \mathcal{O}_X(-mC_{13} - C_{32})$
and $L_1 = \mathcal{O}_X(-mC_{13} - C_{32} - C_{21})$, where $C_{jj'}$ is the line joining $P_j,P_{j'}$.
Then $L_3$ is invertible except at $P_3$, $L_2$ is invertible except at $P_3, P_2$ and dual invertible at $P_3$, 
and $L_1$ is invertible except at $P_3, P_2, P_1$ and dual invertible at $P_3, P_2$.  

Let $F_i$ be the versal NC deformation of $L_i$ for $i = 1,2,3$, 
and let $R_i = \text{End}(F_i)$ be the parameter algebra of the deformation.
We apply the results of the previous subsection on $F$ and $\tilde F$ to $L_i$ and $F_i$, respectively.
Since $L_i$ is locally free or dual free except at $P_i$, 
we have $\text{Ext}^j(L_i,L_i) \cong \mathcal Ext^j(L_i,L_i)_{P_i}$, and 
$F_i$ is locally free or dual free everywhere (Lemmas \ref{local vanishing}, \ref{3.2} and \ref{3.3}).

We have a semi-orthogonal decomposition due to Karmazyn-Kuznetsov-Shinder \cite{KKS}:

\begin{Thm}[\cite{KKS} Example 6.11]
\[
D^b(X) = \langle L_1,L_2,L_3 \rangle = \langle \overline F_1,\overline F_2,\overline F_3 \rangle 
\cong \langle D^b(R_1), D^b(R_2), D^b(R_3) \rangle
\]
where the $\overline F_i$ are the abbreviations of the $\overline{\langle F_i \rangle} \cap D^b(X)$.
\end{Thm}

We give a sketch of the proof.
We note here again that there is a slight difference from \cite{KKS} in the construction of the $F_i$; 
we use anti-effective divisors instead of effective divisors to simplify the notation.

Let $f: Y \to X$ be the minimal resolution, and 
let $E_{i,j}$ ($i = 1,2,3$, $j = 1,\dots,m_i$) be the exceptional curves above $P_i$ such that the curves
\[
C'_{13}, E_{3,m_3}, \dots, E_{3,1}, C'_{32}, E_{2,m_2}, \dots, E_{2,1}, C'_{21}, E_{1,m_1}, \dots, E_{1,1}
\]
form a cycle of $\mathbf P^1$'s on $Y$ in this order, where ${}'$ means the strict transform by $f$.
The sum of these curves belongs to the anti-canonical linear system $\vert -K_Y \vert$.

We define $L_{3,m_3} = \mathcal{O}(-mC'_{13})$, $L_{2,m_2} = L_{3,0}(- C'_{32})$, $L_{1,m_1} = L_{2,0}(- C'_{21})$, 
and $L_{i,j-1} = L_{i,m_i}(-E_{i,m_i}-\dots-E_{i,j})$.
Then the $L_{i,j}$ are exceptional objects on $Y$ and 
\[
\mathbb R\text{Hom}(L_{i,j},L_{i',j'}) = 0, \quad \text{ if } i > i', \text{ or } i = i' \text{ and } j > j'.
\]
Let $\tilde L_i$ ($i=1,2,3$) be the versal $m_i+1$-pointed NC deformation 
of $\oplus_{j=0}^{m_i} L_{i,j}$.
Since they are also $k^{m_i+1}$-modules, we can write
$\tilde L_i = \oplus_{j=0}^{m_i} \tilde L_{i,j}$. 

Then the versal NC deformation $F_i$ of $L_i$ is given as
$F_i = f_*\tilde L_{i,0}$.
As shown above, all $F_i$ are locally free or dual free.
More precisely, $F_3$ is locally free, $F_2$ is locally dual free at $P_3$, and $F_1$ is locally dual free at $P_3$ and 
$P_2$.
The semi-orthogonal decomposition of $D^b(X)$ is a consequence a semi-orthogonal decomposition on $Y$
(\cite{KKS}):
\[
D^b(Y) = \langle L_{1,0},\dots,L_{1,m_1},L_{2,0},\dots,L_{2,m_2},L_{3,0},\dots,L_{3,m_3} \rangle.
\]

%%%%%%%%%%%%%%%%%%%%%%%
%%%%%%%%%%%%%%%%%%%%%%%
%%%%%%%%%%%%%%%%%%%%%%%
\section{Main theorem: Wahl singularity case}

The following is a modification of \cite{H} Theorem 1.1.
We note here again that the sign change from positive to negative makes the result simpler. 

\begin{Thm}\label{Hacking}
Let $X$ be a normal projective variety of dimension $2$ such that $H^p(X,\mathcal{O}_X)= 0$ for $p > 0$.
Assume the following conditions:

(a) There is a quotient singularity $P \in X$ of type $\frac 1{r^2}(1, ar-1)$ for positive integers $a,r$ such that 
$0 < a < r$ and $(r,a) = 1$.

(b) There exists a divisorial sheaf $A = \mathcal{O}_X(-D)$ on $X$ for a Weil divisor $D$ such that 
$D$ is equivalent to a toroidal coordinate axis in an analytic neighborhood of $P$, and that $A$ is 
invertible or dual invertible at other singularities of $X$.

(c) There is a projective flat deformation $f: \mathcal X \to \Delta$ over a disk $\Delta$ with a coordinate $t$
such that $X = f^{-1}(0)$,  
$f^{-1}(t)$ is smooth for $t \ne 0$, and that $\mathcal X$ is $\mathbf Q$-Gorenstein at $P$. 

Then, after replacing $\Delta$ by a smaller disk around $0$ and after a finite base change 
by taking roots of the coordinate $t$, 
there exists a maximally Cohen-Macaulay sheaf $\mathcal E$ of rank $r$ on $\mathcal X$ 
which satisfies the following conditions:

\begin{enumerate}
\item $\mathcal E \otimes \mathcal O_X \cong A^{\oplus r}$.

\item $E := \mathcal E \otimes \mathcal O_Y$ is an exceptional vector bundle on $Y = f^{-1}(t)$ for $t \ne 0$.
\end{enumerate}
\end{Thm}

\begin{proof}
Let $X^! \to X$ be an {\em index $1$ cover} (or {\em canonical cover}) of an analytic neighborhood of $P$.
Then it is a hypersurface singularity defined by 
\[
X^! = \{xy = z^r\} \subset \mathbf C^3
\]
with an action of $G = \mathbf Z/(r)$ given by $(x,y,z) \mapsto (\zeta x, \zeta^{-1} y, \zeta^a z)$, where
$\zeta$ is a primitive $r$-th root of unity. 

In general, if $\mathcal X \to S$ is a $\mathbf Q$-Gorenstein deformation of a $\mathbf Q$-Gorenstein singularity $X$, then 
$K_{\mathcal X/S}$ is also a $\mathbf Q$-Cartier divisor.
An index $1$ cover
$\mathcal X^! \to \mathcal X$ gives a deformation $\mathcal X^! \to S$ of $X^!$ 
with the Galois group action of $G = \text{Gal}(X^!/X) = \text{Gal}(\mathcal X^!/\mathcal X)$.

A versal deformation of $X^!$ is given by
\[
xy = z^r + \sum_{i = 0}^{r-2} t_iz^i
\]
for the parameters $t_i$, and the action of $G$ extends only if $t_i = 0$ for $i > 0$.
Therefore a versal $\mathbf Q$-Gorenstein deformation of $X$ is given 
\[
\{xy = z^r + t\} \subset \frac 1r(1,-1,a,0).
\]

After replacing $\Delta$ by a finite base change, 
our deformation becomes a pull-back of a deformation
\[
\mathcal X = \{xy=z^r+t^a\} \subset \frac 1r(1,ar-1,a,0) = \frac 1r(1,ar-1,a,r)
\]
which we denote by the same letter and treat in the following.

Following \cite{H} \S 3, 
let $\mu: \mathcal X' \to \mathcal X$ be a weighted blow up at $P$ with weights $\frac 1r(1,ar-1,a,r)$, and 
let $W$ be the exceptional divisor.
Thus $W \cong \{xy = z^r + t^a\} \subset \mathbf P(1,ar-1,a,r)$.
Let $X'$ be the strict transform of $X = f^{-1}(0)$.
Then the induced morphism $\mu \vert_{X'}: X' \to X$ is a weighted blow-up with weights $\frac 1{r^2}(1,ar-1)$, where 
we have $(x,y,z) = (u^r, v^r, uv)$ for semi-invariant coordinates $(u,v)$ with weights $\frac 1{r^2}(1,ar-1)$.
We have $\mu^{-1}f^{-1}(0) = X' \cup W$.
We denote $C = X' \cap W$.
It is a smooth rational curve.

By \cite{H} Proposition 5.1, there exists an exceptional vector bundle $G$ of rank $r$ on $W$ such that 
$G \otimes \mathcal O_C \cong \mathcal{O}_C(-1)^{\oplus r}$.
Here we note that we take a dual bundle $F_2^*$ in \cite{H} as $G$.

Let $A'$ be the strict transform of $A$ on $X'$.
Then it is invertible or dual invertible, and invertible near $C$ 
such that $A' \otimes \mathcal O_C \cong \mathcal{O}_C(-1)$.
Then by gluing $G$ with $(A')^{\oplus r}$, we obtain a locally free or dual free sheaf 
$G'$ on $X' \cup W$.

By \cite{H} Proposition 5.1, we have $H^i(W,G) \cong 0$ for all $i$.
Here we note that such vanishings do not hold for $G^*$.
We have an exact sequence on $X' \cup W$:
\[
0 \to G' \to G \oplus (A')^{\oplus r} \to \mathcal O_C(-1)^{\oplus r} \to 0.
\]
Since $R\mu_*G \cong R\mu_*O_C(-1) \cong 0$, we have $R\mu_*G' \cong R\mu_*(A')^{\oplus r}$. 

We claim that $R(\mu \vert_{X'})_*A' \cong A$.
Since it is a local assertion near $C \cong \mathbf P^1$, 
we may assume that $A' = \mathcal{O}_{X'}(-l')$ for a smooth curve $l'$ in a small neighborhood of $C$ which 
intersects $C$ transversally.
Then $l = \mu \vert_{X'}(l')$ is again a smooth curve and $R(\mu \vert_{X'})_*\mathcal{O}_{l'} \cong \mathcal{O}_l$, hence
$R(\mu \vert_{X'})_*A' \cong \mathcal{O}_X(-l) \cong A$.
It follows that we have $R\mu_*G' \cong R\mu_*(A')^{\oplus r} \cong R(\mu \vert_{X'})_*(A')^{\oplus r} \cong A^{\oplus r}$.

We calculate the endomorphism sheaf $\mathcal End(G')$.
Since $A'$ is a divisorial sheaf on $X'$, we have $\mathcal End(A') \cong \mathcal{O}_{X'}$.
Since $G$ is a locally free sheaf on $W$, we obtain $\mathcal End(G')$ by gluing 
$\mathcal End(G)$ on $W$ and $\mathcal End((A')^{\oplus r}) \cong \mathcal{O}_{X'}^{\oplus r^2}$ on $X'$.
Thus we have an exact sequence
\[
0 \to \mathcal End(G') \to \mathcal End(G) \oplus \mathcal{O}_{X'}^{\oplus r^2} 
\to \mathcal{O}_C^{\oplus r^2} \to 0.
\]
Since $G$ is an exceptional vector bundle on $W$, we have 
$\mathbb R\Gamma(W, \mathcal End(G)) \cong \mathbb R\text{Hom}_W(G,G) \cong k$.
Since $A'$ is invertible or dual invertible, we have $\mathcal Ext^i(A',A') \cong 0$ for $i > 0$.
We deduce that 
\[
\mathbb R\text{Hom}_{X' \cup W}(G',G') \cong \mathbb R\Gamma(X' \cup W,\mathcal End(G')) \cong k.
\]
It follows that $G'$ deforms to yield a locally free or dual free sheaf $\mathcal E'$ on $\mathcal X'$
by Proposition \ref{extension}.

We set $\mathcal E = \mu_*\mathcal E'$.
It is a torsion free sheaf on $\mathcal X$ which is locally free or dual free except at the point $\mu(W)$.
Since $R\mu_*G' \cong A^{\oplus r}$, we have $R^i\mu_*\mathcal E' = 0$ for $i > 0$ by 
the upper semi-continuity theorem.
Indeed, if we take a sufficiently ample sheaf $H$ on $\mathcal X$, then we have
$H^i(X' \cup W, G' \otimes \mu^*H) \cong 0$ for $i > 0$.
It follows that $H^0(\mathcal X, R^i\mu_* \mathcal E' \otimes H) \cong H^i(\mathcal X', \mathcal E' \otimes \mu^*H) \cong 0$ for $i > 0$, 
and $R^i\mu_*\mathcal E' = 0$ for $i > 0$.

It also follows that a natural homomorphism $\mu_*\mathcal E' \to \mu_*G'$ is surjective. 
Hence we have $\mathcal E \otimes \mathcal{O}_X \cong A^{\oplus r}$, and 
$\mathcal E$ is a maximally Cohen-Macaulay sheaf, since so is $A$.

Since $\mathcal E \otimes \mathcal O_Y \cong \mathcal E' \otimes \mathcal O_Y$ on a general fiber $Y$, 
we have $\mathbb R\text{End}_Y(\mathcal E \otimes \mathcal O_Y) \cong k$ by the upper semi-continuity theorem.
We note here that $\mathcal End(\mathcal E)$ is flat over $\Delta$
because $\mathcal Hom^i(\mathcal E,\mathcal E) \cong 0$ for $i > 0$ and $\mathcal E$ is flat.
\end{proof}

\begin{Rem} 
In the above theorem, a global topological condition on $X$ in \cite{H} is replaced by the assumption on the 
existence of $A$.
A divisorial sheaf, say $A_H$, considered in \cite{H} is locally isomorphic to 
$\mathcal{O}(1)$ instead of $\mathcal{O}(-1)$, and the 
reflexive sheaf, say $\mathcal E_H$, on $\mathcal X$ satisfies 
that $(\mathcal E_H \otimes \mathcal O_X)^{**} \cong (A_H)^{\oplus r}$, i.e., we need to 
take a double dual. 
This is avoided by using negative degree sheaf $A$. 
Our $\mathcal E$ is equal to the dual $\mathcal E_H^*$.

We allowed that $X$ has singularities other than $P$ unlike in \cite{H}, but the same proof works for the construction of 
$\mathcal E$ and the proof of its stability.
We assumed that $X$ has a smoothing which is $\mathbf Q$-Gorenstein at $P$, 
and that $A$ is invertible or dual invertible except at $P$ so that there is no local deformation of $A$.
\end{Rem}

We prove that a pretilting bundle coming from a NC deformation on a special fiber 
deforms to a direct sum of Hacking's bundle on a generic fiber under a $\mathbf Q$-Gorenstein smoothing: 

\begin{Thm}\label{direct sum}
Assume the conditions of Theorem \ref{Hacking} and use the notation there.
Let $F$ be a versal NC deformation of $A$ on $X$.
Then the following hold.

(0) $F$ is a locally free or dual free sheaf of rank $r^2$ and is locally free at $P$.

(1) $\mathbb R\text{Hom}_X(F,A) \cong k$.
In particular, $F$ is pretilting.

(2) $\text{Ext}^i_X(F,F) \cong 0$ for $i > 0$, and $F$ extends to a locally free or dual free sheaf $\mathcal F$ on $\mathcal X$, 
if $\Delta$ is replaced by a smaller disk (it is not necessary to replace $\Delta$ by its covering).

(3) $\mathcal F \otimes \mathcal O_Y \cong (\mathcal E \otimes \mathcal O_Y)^{\oplus r}$ on $Y$.
In particular $\text{End}(F)$ deforms to $\text{Mat}(k,r)$.
\end{Thm}

\begin{proof}
(0) is already proved by Lemmas \ref{3.2} and \ref{3.3}.

(1) Since $F$ is the versal NC deformation of $A$, we have $\text{Hom}_X(F,A) \cong k$
and $\text{Ext}^1_X(F,A) = 0$ by the construction (\cite{NC multi}).

Since $F$ is locally free at $P$ and locally free or dual free elsewhere, the higher
extension sheaves vanish: $\mathcal Ext^i(F,A) = 0$ for $i > 0$.
Therefore we have 
\[
\text{Ext}^2_X(F,A) \cong H^2(X, \mathcal Hom(F,A)) \cong H^0(X, (A^* \otimes F \otimes K_X)^{**})^*
\]
by the Serre duality.
Since $F$ is a successive extension of $A$ and $H^0(X,K_X) = 0$, we deduce that $\text{Ext}^2(F,A) = 0$.
Thus $\mathbb R\text{Hom}_X(F,A) = k$.

(2) It follows from (1) that $\text{Ext}^i_X(F,F) = 0$ for $i > 0$.
By Proposition \ref{extension}, 
$F$ deforms flatly to a locally free or dual free sheaf $\mathcal F$ on $\mathcal X$, if we shrink $\Delta$
if necessary.

(3) By (1) and the upper semi-continuity theorem, 
we obtain 
\[
\mathbb R\text{Hom}_Y(\mathcal F \otimes \mathcal O_Y,\mathcal E \otimes \mathcal O_Y) \cong k^{\oplus r}.
\]
We prove that a natural homomorphism 
\[
\mathcal F \otimes \mathcal O_Y \to \text{Hom}(\mathcal F \otimes \mathcal O_Y,\mathcal E \otimes \mathcal O_Y)^* 
\otimes \mathcal E \otimes \mathcal O_Y 
\cong (\mathcal E \otimes \mathcal O_Y)^{\oplus r}
\]
is an isomorphism of sheaves.

By \cite{H} Theorem 1.1 and Proposition 4.4, $\mathcal E \otimes \mathcal O_Y$ is slope stable
with respect to any ample line bundle.
Since $F$ is a successive extension of a divisorial sheaf $A$, it is slope semistable.
Indeed, suppose that there is a subsheaf $B \subset F$ which attains the maximal slope $\mu(B) > \mu(F)$.
Since $F$ is a successive extension of $A$, there is a non-zero homomorphism $h: B \to A$.
Since $\mu(B) > \mu(F) = \mu(A)$, it follows that 
$\mu(\text{Ker}(h)) > \mu(B)$, a contradiction to the maximality of $B$. 
Therefore $\mathcal F \otimes \mathcal O_Y$ is also semistable because semistability is an open condition.

Let $f_1,\dots,f_r \in \text{Hom}(\mathcal F \otimes \mathcal O_Y,\mathcal E \otimes \mathcal O_Y)$ be a basis.
We introduce a decreasing filtration $K^p = \bigcap_{i=1}^p \text{Ker}(f_i)$ of $\mathcal F \otimes \mathcal O_Y$ for $0 \le p \le r$, 
where we set $K^0 = \mathcal F \otimes \mathcal O_Y$.
We will prove that $f_{p+1} \vert_{K^p}: K^p \to \mathcal E \otimes \mathcal O_Y$ is surjective in codimension $1$, 
$\text{rank}(K^{p+1}) = r(r-p-1)$, 
and that $\mu(K^{p+1}) = \mu(A)$ by induction on $p$.

Assume that our assertion is already proved for $p < p_0$ for some integer $p_0 \ge 0$.
First assume that $f_{p_0+1} \vert_{K^{p_0}} \ne 0$.
Since $\mathcal F \otimes \mathcal O_Y$ is semistable 
and $\mathcal E \otimes \mathcal O_Y$ is stable with the same slope, it follows that $K^{p_0}$ is also semistable with 
$\mu(K^{p_0}) = \mu(\mathcal E \otimes \mathcal O_Y) = \mu(A)$.
It follows that $f_{p_0+1} \vert_{K^{p_0}}$ is surjective in codimension $1$ by the stability of $\mathcal E \otimes \mathcal O_Y$.
Then $\text{rank}(K^{p_0+1}) = r(r-p_0-1)$ and $\mu(K^{p_0+1}) = \mu(A)$.

Now we will prove that $f_{p_0+1} \vert_{K^{p_0}} \ne 0$.
Assuming that $f_{p_0+1} \vert_{K^{p_0}} = 0$, we claim that there exist $c_i \in k$ for $1 \le i \le p_0$ such that 
$f_{p_0+1} = \sum_{i = 1}^{p_0} c_if_i$, a contradiction with the linear independence. 
We will prove this claim by descending induction on $i$.
Assume that the $c_i$ for $p_1 < i \le p_0$ are already determined for an integer $p_1$ with $p_1 \le p_0$ so that 
$(f_{p_0+1} - \sum_{i=p_1+1}^{p_0} c_if_i) \vert_{K^{p_1}} = 0$.
We consider homomorphisms $(f_{p_0+1} - \sum_{i=p_1+1}^{p_0} c_if_i) \vert_{K^{p_1-1}}$ and 
$f_{p_1} \vert_{K^{p_1-1}}$ from $K^{p_1-1}$ to $\mathcal E \otimes \mathcal O_Y$.
Since $\text{End}(\mathcal E \otimes \mathcal O_Y) \cong k$, there exists $c_{p_1} \in k$ such that 
$(f_{p_0+1} - \sum_{i=p_1}^{p_0} c_if_i) \vert_{K^{p_1-1}} = 0$.
Thus the existence of the $c_i$ is proved and hence the property of the filtration is proved.

It follows that $\text{Gr}_K(\mathcal F \otimes \mathcal O_Y) = \bigoplus_{i=1}^r K^{i-1}/K^i$ 
is a subsheaf of $(\mathcal E \otimes \mathcal O_Y)^{\oplus r}$.
Since both sheaves have the same rank and same slope, the support of their quotients has dimension $0$.
We have $\chi(\mathcal F \otimes \mathcal O_Y) = \chi(F) = r^2 \chi(A) = r \chi(\mathcal E \otimes \mathcal O_Y)$, hence 
$\text{Gr}_K(\mathcal F \otimes \mathcal O_Y) \cong (\mathcal E \otimes \mathcal O_Y)^{\oplus r}$.
Since $\text{Ext}^1(\mathcal E \otimes \mathcal O_Y,\mathcal E \otimes \mathcal O_Y) = 0$, we obtain our assertion.
\end{proof}

%%%%%%%%%%%%%%%%%%%%%%%%%%%%%%
%%%%%%%%%%%%%%%%%%%%%%%%%%%%%%
%%%%%%%%%%%%%%%%%%%%%%%%%%%%%%
\section{Main theorem: higher Milnor number case}

In order to generalize our main results to higher Milnor number case (i.e., $s > 1$), 
we use crepant simultaneous partial resolutions and flops between them
which are explained below.
We prove that semi-orthogonality plus flops implies full orthogonality.

We construct a {\em crepant simultaneous partial resolution} of a $\mathbf Q$-Gorenstein smoothing  
$f: \mathcal X \to \Delta$ of a quotient singularity of type $\frac 1{r^2s}(1,ars - 1)$ (cf. \cite{crepant}):

\begin{Lem}\label{crepant}
Let $X$ be a quotient singularity of type $\frac 1{r^2s}(1,ars - 1)$.
Let $f: \mathcal X \to \Delta$ be a $\mathbf Q$-Gorenstein smoothing of $X$.
Then, after a suitable shrinking and finite base change of $\Delta$, 
there exists a birational morphism $\mu: \mathcal X' \to \mathcal X$ which satisfies
the following conditions:

(1) $\mu^{-1}f^{-1}(t) \to f^{-1}(t)$ is an isomorphism for $t \ne 0$, where $t$ is a coordinate on $\Delta$.

(2) $X' = \mu^{-1}f^{-1}(0)$ is a normal surface having $s$ quotient singular points 
$P_1,\dots,P_s$ of type $\frac 1{r^2}(1,ar-1)$, and 
$f':= f \circ \mu: \mathcal X' \to \Delta$ is a $\mathbf Q$-Gorenstein smoothing of $X'$.

(3) $\mu: X' \to X$ is {\em crepant}, i.e., $K_{X'} = \mu^*K_X$ as $\mathbf Q$-Cartier divisors.

(4) The exceptional curves $C_1,\dots,C_{s-1} \subset X'$ of $\mu$ form a chain of $\mathbf P^1$'s 
connecting $s$ singular points such that $P_i,P_{i+1} \in C_i$ for all $i$.
\end{Lem}

\begin{proof}
We embed $X$ as 
$X = \{xy = z^{rs}\} \subset \frac 1r(1,-1,a)$ by $x=u^{rs}$, $y=v^{rs}$, $z = uv$.
An index $1$ cover $X^!$, or a canonical cover, of $X$ is given by an equation
\[
\{xy = z^{rs}\} \subset \mathbf C^3
\]
with an action of $G = \mathbf Z/(r)$ given by 
$(x,y,z) \mapsto (\zeta x, \zeta^{-1}y, \zeta^a z)$.
The versal deformation of $X^!$ is given by 
$xy = z^{rs} + \sum_{i=0}^{rs-2} t_iz^i$, where the $t_i$ are parameters.
Hence the equation of a versal $\mathbf Q$-Gorenstein deformation of $X$ is given by 
\[
xy = z^{rs} + \sum_{i=0}^{s-1} t_iz^{ir}
\]
since it should be invariant under the Galois group action. 

Any one parameter deformation of $X$ is given by an equation
$xy = z^{rs} + \sum_{i=0}^{s-1} g_i(t)z^{ir}$, where $t$ is the parameter on $\Delta$ and the $g_i$ are holomorphic functions
such that $g_i(0) = 0$. 
After a suitable base change $t \mapsto t^m$ and shrinking of $\Delta$, we obtain factorization
of the equation of $\mathcal X$:
\[
f: \mathcal X = \{xy = \prod_{i=1}^s (z^r-h_i(t))\} \subset V = \frac 1r(1,-1,a,0) \to \Delta
\]
where the $h_i$ are holomorphic functions such that $h_i(0) = 0$.

We construct $\mu$ by induction on $s$.
If $s = 1$, then $\mu$ is the identity.
Assume that $s > 1$ in the following.

Let $\mu_1: \mathcal X_1 \to \mathcal X$ be a blow up along the ideal $(x,z^r-h_1(t))$, 
and let $C_1$ be the exceptional curve.
Then $\mu_1$ is crepant because it is small.
$\mathcal X_1$ is covered by two open subsets $U_1$ and $U_2$.
$U_1$ has an equation 
\[
U_1 = \{x'y = \prod_{i=2}^s (z^r - h_i(t))\} \subset \frac 1r(1,-1,a,0)
\]
with coordinates $(x'=x/(z^r-h_1(t)),y,z,t)$, thus we obtain the same situation with $s$ decreasing by $1$.
$U_2$ has an equation 
\[
U_2 = \{y = \prod_{i=1}^s (z^r-h_i(t))/x, xt' = z^r-h_1(t)\} \subset \frac 1r(1,-1,a,0,-1)
\]
with coordinates $(x,y,z,t,t')$.
In other words, 
\[
U_2 = \{xt' = z^r-h_1(t)\} \subset \frac 1r(1,-1,a,0)
\]
with coordinates $(x,t',z,t)$, and we obtain a situation with $s = 1$. 

Therefore $\mu_1^{-1}f^{-1}(0)$ has two singular points of types $\frac 1{r^2}(1,ar-1)$ and 
$\frac 1{r^2(s-1)}(1,ar(s-1)-1)$.
By repeating the above small blow ups, we obtain our $\mu$.
\end{proof}

We note that $\mu: X' \to X$ does not factor the minimal resolution of $X$.
For example, if $r = s = 2$, then the exceptional locus of the minimal resolution $\nu: X'' \to X$ of the quotient singularity of type 
$\frac 18(1,3)$ consists of two $(-3)$-curves $E_1,E_2$.
$X'$ is obtained from $X''$ by blowing up $E_1 \cap E_2$ then contracting the strict transforms $E'_1,E'_2$ which are 
$(-4)$-curves.

Since $\mu: \mathcal X' \to \mathcal X$ is crepant, we can flop curves $C_i$ on $\mathcal X'$ for $1 \le i \le s-1$ (\cite{crepant}).
Let $\mu_i: \mathcal X'_i \to \mathcal X$ be the flopped morphism.
Though $\mathcal X'$ and $\mathcal X'_i$ are isomorphic, 
the natural birational map $\mathcal X' \dashrightarrow \mathcal X'_i$ is not a morphism.
It induces a birational map $X' \dashrightarrow X'_i$ which is extendable to an isomorphism by the dimension reason.
 
A flop acts on the set of divisors on $\mathcal X'$ because the natural birational map $\mathcal X' \dashrightarrow \mathcal X'_i$ 
is an isomorphism in codimension $1$.
Because the Milnor fiber of the Wahl singularity is a $\mathbf Q$-homology ball (\cite{W}, \cite{moderate}), 
the cohomology groups of the fibers with coefficients in $\mathbf Q$ are constant.
Hence this action induces an action on the numerical classes of divisors on $X'$ by the restriction of 
$\mathbf Q$-Cartier divisors.

We calculate this action using the following lemma:

\begin{Lem}\label{mu flop}
Let $\mu: V \to W$ be a projective birational morphism of $3$-dimensional varieties with only terminal singularities
whose exceptional locus is an irreducible curve $C$ such that $(K_V, C) = 0$, and let 
$\mu': V' \to W$ be its flop with the exceptional curve $C'$.
Let $D$ be a $\mathbf Q$-Cartier divisor on $V$ and let $D'$ be its strict transform on $V'$.
Then $(D',C') = - (D,C)$.
\end{Lem}

\begin{proof} 
We use the construction of flops in \cite{flop}.
Since $\mu$ is crepant, $W$ has only terminal singularities too.
We can replace $W$ by its analytic germ around $\mu(C)$, because the intersections occur above this germ.
We can also replace $W$ by its index $1$ cover and replace $V$ and $V'$ by their pull-backs, 
because the equality of intersection numbers is preserved by a finite covering.
Then $W$ becomes a hypersurface singularity of multiplicity $2$, a double cover of a smooth germ.
There is a Galois involution $\sigma: W \to W$ which underlies  
the flop $(\mu')^{-1}\mu: V \dashrightarrow V'$.
More precisely, we have $V' = V \times_W W^{\sigma}$, where the symbol $W^{\sigma}$ means that 
the map $W^{\sigma} \to W$ is given by $\sigma$.  
$E := \mu_*D + \sigma_*\mu_*D$ is a pull-back of a $\mathbf Q$-divisor on the smooth germ, hence is a 
$\mathbf Q$-Cartier divisor.

The isomorphism $\sigma$ induces an isomorphism $\sigma': V \to V'$.
We have $\mu_*(\sigma')^*D' = \sigma_*\mu_*D$, hence $D + (\sigma')^*D' = \mu^*E$.
Since $C$ is contracted by $\mu$ and $\sigma'(C) = C'$, we have $0 = (D, C) + ((\sigma')^*D', C) = (D, C) + (D',C')$,   
hence the result.
\end{proof} 

We continue to use the notation of Lemma \ref{crepant}.

\begin{Lem}
Let $C_0,C_s$ be strict transforms on $X'$ of the curves corresponding to the two coordinate axes on $X$
such that $C_0,C_1,\dots,C_{s-1},C_s \subset X'$ form a chain of curves in this order, so that 
$C_i \cap C_j = \emptyset$ if $\vert i - j \vert \ge 2$.
Then the following hold.

(1) $(C_{i-1}, C_i) = 1/r^2$ for $1 \le i \le s$, and $(C_i^2) = - 2/r^2$ for $1 \le i \le s-1$.

(2) The flop of the curve $C_i$ interchanges $C_0 + \dots + C_{i-1}$ and $C_0 + \dots + C_i$ for $1 \le i \le s-1$.
\end{Lem}

\begin{proof}
(1) The first equality follows since the order of the quotient singularity is $1/r^2$. 
We have $(K_{X'}, C_i) = 0$ for $1 \le i \le s-1$, and 
$K_{X'} + C_0 + C_1 + \dots C_{s-1} + C_s \sim 0$.
Then $\sum_{k=0}^s (C_k, C_i) = 0$, hence the second equality.

(2) For $1 \le i,j \le s-1$, we have
\[
((C_0 + \dots + C_j),C_i) = \begin{cases} 1/r^2, &\quad j = i-1, \\
-1/r^2, &\quad j = i, \\
0, &\quad j \ne i-1,i.  \end{cases} 
\]
Hence the assertion follows from Lemma \ref{mu flop}.
\end{proof}

\begin{Thm}\label{Hacking'}
Let $X$ be a normal projective variety of dimension $2$ such that $H^p(X,\mathcal{O}_X)= 0$ for $p > 0$.
Assume the following conditions:

(a) There is a quotient singularity $P \in X$ of type $\frac 1{r^2s}(1, ars-1)$ for positive integers $a,r,s$ such that 
$0 < a < r$ and $(r,a) = 1$.

(b) There exists a divisorial sheaf $A = \mathcal{O}_X(-D)$ on $X$ for a Weil divisor $D$ such that 
$D$ is equivalent to the first coordinate axis with respect to the toroidal coordinate ($\mu(C_0)$ in the above discussion)
in an analytic neighborhood of $P$, and that $A$ is 
locally invertible or dual invertible at other singularities of $X$.

(c) There is a projective flat deformation $f: \mathcal X \to \Delta$ over a disk $\Delta$ with a coordinate $t$
such that $X = f^{-1}(0)$,  
$f^{-1}(t)$ is smooth for $t \ne 0$, and that $\mathcal X$ is $\mathbf Q$-Gorenstein at $P$.

Then, after replacing $\Delta$ by a smaller disk around $0$ and after a finite base change 
by taking roots of the coordinate $t$, 
there exist maximally Cohen-Macaulay sheaves $\mathcal E_1,\dots \mathcal E_s$ of rank $r$ on $\mathcal X$
which satisfy the following conditions:

\begin{enumerate}
\item $\mathcal E_i \otimes \mathcal O_X \cong A^{\oplus r}$ for all $i$.

\item $E_i := \mathcal E_i \otimes \mathcal O_Y$ are exceptional vector bundles on $Y = f^{-1}(t)$ for $t \ne 0$, 
which are mutually orthogonal, i.e., 
\[
\mathbb R\text{Hom}_Y(E_i,E_j) := \bigoplus_{p=0}^2 \text{Ext}^p(E_i,E_j)[-p] \cong 0
\]
for $i \ne j$.
\end{enumerate}
\end{Thm}

\begin{proof}
Let $\mu: \mathcal X' \to \mathcal X$ be the crepant simultaneous partial resolution 
constructed in Lemma \ref{crepant}
The central fiber $X'$ has $s$ singularities $P_i$ of type $\frac 1{r^2}(1,ar-1)$, and 
$\mathcal X'$ is a $\mathbf Q$-Gorenstein smoothing of $X'$ whose general fiber is the same 
as that of the smoothing $\mathcal X$.
The exceptional curves $C_i$ connect the singular points of $X'$ as described in the lemma.

Let $A'$ be a divisorial sheaf on $X'$ corresponding to the strict transform of the divisor $D$ in the following way:
$\mu_*A' \cong A$ and that $A'$ is locally isomorphic to $\mathcal{O}_{X'}(-C_0)$ near $\mu^{-1}(P)$.
Let $A'_i = A'(- C_1 - \dots - C_{i-1})$ for $1 \le i \le s$.
Then $A'_i$ is a divisorial sheaf on $X'$ which is locally isomorphic to $\mathcal{O}_{X'}(-C_{i-1})$ near $P_i$, 
locally dual invertible at $P_j$ for $j < i$, and locally invertible elsewhere.
 
By Theorem \ref{Hacking}, there is a maximally Cohen-Macaulay sheaf $\mathcal E'_i$ on $\mathcal X'$ 
for every $1 \le i \le s$ such that 
$\mathcal E'_i \otimes \mathcal O_{X'} \cong (A'_i)^{\oplus r}$ and $\mathcal E'_i \otimes \mathcal O_Y$ is an 
exceptional vector bundle.

We consider $A'_i$ and $A'_j$ such that $i < j$.
Then $A'_i$ (resp. $A'_j$) is locally dual invertible at the points $P_1,\dots,P_{i-1}$ 
(resp. $P_1, \dots, P_{j-1}$) and locally invertible elsewhere except at $P_i$ (resp $P_j$).
It follows that $\mathcal Ext^k(A'_i,A'_j) \cong 0$ for $k > 0$ by Lemma \ref{local vanishing}. 
We calculate 
\[
\begin{split}
&\mathbb R\text{Hom}_{X'}(A'_i, A'_j) \cong \mathbb R\Gamma(X', \mathcal Hom(A'_i,A'_j)) \\
&\cong \mathbb R\Gamma(X',\mathcal{O}_{X'}(-(C_i + \dots + C_{j-1})) \cong 0
\end{split}
\]
because $\mathbb R\Gamma(X',\mathcal{O}_{X'}) \cong \mathbb R\Gamma(X', \mathcal{O}_{C_i + \dots + C_{j-1}}) \cong k$.
By the upper semi-continuity, we obtain
\[
\mathbb R\text{Hom}_Y(\mathcal E'_i \otimes \mathcal O_Y,\mathcal E'_j \otimes \mathcal O_Y) \cong 0
\]
for $i < j$.

If we flop a curve $C_i$ on $\mathcal X'$, the divisorial sheaves $A'_i$ and $A'_{i+1}$ on $X'$ are
interchanged.
More precisely, let $\alpha: \mathcal X' \dashrightarrow \mathcal X'_1$ be the flop of a curve $C_i$.
Then there is an isomorphism of special fibers $\beta: X' \to X'_1$ which is induced by the flop as a birational map, but not the restriction of the flop.
$\alpha$ induces an isomorphism $\alpha_*$ of the spaces of divisors from $\mathcal X'$ to $\mathcal X'_1$, and the latter induces by restriction 
an isomorphism of the space of divisors on the special fibers.
By pulling back by $\beta$, we obtain the above correspondence of divisors on $X'$.

But since the generic fiber is unchanged under the flop, we obtain 
\[
\mathbb R\text{Hom}_Y(\mathcal E'_i \otimes \mathcal O_Y,\mathcal E'_j \otimes \mathcal O_Y) \cong 0
\]
for $i > j$.
Therefore we have 
\[
\mathbb R\text{Hom}_Y(\mathcal E'_i \otimes \mathcal O_Y,\mathcal E'_j \otimes \mathcal O_Y) \cong 0
\]
whenever $i \ne j$.

We claim that $R\mu_*A'_i \cong A$, i.e., $R^j\mu_*A'_i \cong 0$ for $j > 0$ and all $i$.
Indeed, since it is a local assertion, we may assume that 
$X = \frac 1{r^2s}(1,ars-1)$, $A' = \mathcal{O}_{X'}(-C_0)$ and $A = \mathcal{O}_X(-\mu(C_0))$.
Then we have $R\mu_*\mathcal{O}_{C_0+C_1 + \dots + C_{i-1}} \cong \mathcal{O}_{\mu(C_0)}$.
Since $R\mu_*\mathcal{O}_{X'} \cong \mathcal{O}_X$, we obtain our claim.

By the upper semi-continuity again as in the proof of Theorem \ref{Hacking}, 
we have $R^j\mu_*\mathcal E'_i \cong 0$ for $j > 0$ and all $i$.
We define $\mathcal E_i$ by $R\mu_*\mathcal E'_i = \mathcal E_i$.
We have an exact sequence
\[
0 \to \mathcal E'_i \to \mathcal E'_i \to (A'_i)^{\oplus r} \to 0
\]
where the first arrow is the multiplication of $t$.
Then we have 
\[
0 \to \mathcal E_i \to \mathcal E_i \to A^{\oplus r} \to 0
\]
i.e., $\mathcal E_i \otimes \mathcal O_X \cong A^{\oplus r}$.
This completes the proof.
\end{proof}

\begin{Thm}\label{direct sum'}
Assume the conditions of Theorem \ref{Hacking'}.
Let $F$ be a versal NC deformation of $A$ on $X$.
Then the following hold.

(0) $F$ is a locally free or dual free sheaf of rank $r^2s$ and is locally free at $P$.

(1) $\mathbb R\text{Hom}_X(F,A) = k$.

(2) $\text{Ext}^i_X(F,F) = 0$ for $i > 0$, and $F$ deforms to a locally free or dual free sheaf $\mathcal F$, 
if $\Delta$ is replaced by a smaller disk.

(3) $\mathcal F \otimes \mathcal O_Y \cong \bigoplus_{i=1}^s (\mathcal E_i \otimes \mathcal O_Y)^{\oplus r}$ on $Y$.
In particular $\text{End}(F)$ deforms to $\text{Mat}(k,r)^{\times s}$.
\end{Thm}

\begin{proof}
(0) is already known.
The proofs (1) and (2) are the same as in Theorem \ref{direct sum}.

(3) We modify the proof of Theorem \ref{direct sum}.
We have again $\mathbb R\text{Hom}_Y(\mathcal F \otimes \mathcal O_Y,\mathcal E_i \otimes \mathcal O_Y) \cong k^{\oplus r}$. 
We prove that a natural homomorphism 
\[
\mathcal F \otimes \mathcal O_Y \to 
\bigoplus_{i=1}^s \text{Hom}(\mathcal F \otimes \mathcal O_Y,\mathcal E_i \otimes \mathcal O_Y)^* 
\otimes \mathcal E_i \otimes \mathcal O_Y \cong \bigoplus_{i=1}^s (\mathcal E_i \otimes \mathcal O_Y)^{\oplus r}
\]
is an isomorphism.

Since $\text{Hom}(\mathcal E_i \otimes \mathcal O_Y,\mathcal E_j \otimes \mathcal O_Y) \cong 0$ for $i \ne j$, 
we have again a filtration of $\mathcal F \otimes \mathcal O_Y$ such that $\text{Gr}(\mathcal F \otimes \mathcal O_Y)$ 
is a subsheaf of $\bigoplus_{i=1}^s (\mathcal E_i \otimes \mathcal O_Y)^{\oplus r}$ whose cokernel is supported at isolated points.
We have $\chi(\mathcal F \otimes \mathcal O_Y) = \chi(F) = r^2s \chi(A) = 
\sum_{i=1}^s r \chi(\mathcal E_i \otimes \mathcal O_Y)$, hence 
$\text{Gr}(\mathcal F \otimes \mathcal O_Y) \cong \bigoplus_{i=1}^s (\mathcal E_i \otimes \mathcal O_Y)^{\oplus r}$.
Since $\text{Ext}^1(\mathcal E_i \otimes \mathcal O_Y,\mathcal E_j \otimes \mathcal O_Y) = 0$ for all $i,j$, 
we obtain our assertion.
\end{proof}

\begin{Rem}
The referee remarked that Lemma 5.1 and Theorems 5.2 and 5.3 hold true also in the case 
where $X$ has Gorenstein singularities $\frac 1s(1,s-1)$.
This can be considered as the case where $r = a = 1$, and the same proofs work.
Lemma 5.1 is reduced to the well-known simultaneous resolution of Du Val singularities.
A divisorial sheaf on $X$ extends to its smoothing as a divisorial sheaf in Theorem 5.2 because we have a global assumption
$H^2(X, \mathcal O_X) = 0$.
\end{Rem}

%%%%%%%%%%%%%%%%%%%%%%%
\section{Example: $\mathbf Q$-Gorenstein smoothings of weighted projective planes}

We consider $\mathbf Q$-Gorenstein smoothings of weighted projective planes to del Pezzo surfaces
as examples of the main results. 
First we consider $\mathbf Q$-Gorenstein smoothings to $\mathbf P^2$.

\begin{Thm}\label{Pa1a2a3}
Let $a_1,a_2,a_3$ be positive integers satisfying a Markov equation $a_1^2 +a_2^2+a_3^2=3a_1a_2a_3$.
Let $X = \mathbf P(a_1^2,a_2^2,a_3^2)$ be a weighted projective plane, and 
let $D^b(X) = \langle L_1,L_2,L_3 \rangle = \langle \overline F_1,\overline F_2,\overline F_3 \rangle$ 
be the semi-orthogonal decomposition explained in \S 3, 
where $\text{rank}(F_i) = a_i^2$.
Then under a $\mathbf Q$-Gorenstein smoothing $\mathcal X \to \Delta$ of $X$ to $\mathbf P^2$, 
$F_1,F_2,F_3$ deform to a direct sum $E_1^{\oplus a_1}, E_2^{\oplus a_2}, E_3^{\oplus a_3}$
for a full exceptional collection of vector bundles $(E_1,E_2,E_3)$ on $\mathbf P^2$ 
such that $\text{rank}(E_i) = a_i$.
\end{Thm}

\begin{proof}
We note that the $3$ coordinate points of $X$ are automatically quotient singularities of the type $\frac 1{r^2}(1,ar-1)$ or possibly smooth.
By Theorem \ref{direct sum}, $F_i$ deforms to $E_i^{\oplus a_i}$ for an exceptional vector bundle $E_i$.
Since $\mathbb R\text{Hom}(F_i,F_j) \cong 0$ for $i > j$, it follows that 
$\mathbb R\text{Hom}(E_i,E_j) \cong 0$ for $i > j$ by the upper semi-continuity theorem.
Note that the same conclusion holds even if one of the $a_i$ are equal to $1$.

We will prove that the exceptional collection $(E_1,E_2,E_3)$ is full.
The following proof is due to the suggestion of the referee.
The point is that the fullness is an open property because the supports of coherent sheaves are closed.
Let $\mathcal F_i$ be the locally free or dual free sheaves on $\mathcal X$ obtained in Theorem 4.3
such that 
$\mathcal F_i \otimes \mathcal O_X \cong F_i$ and $\mathcal F_i \otimes \mathcal O_Y \cong E_i^{\oplus a_i}$
for $i = 1,2,3$.
Let $\mathcal R_i = \text{End}(\mathcal F_i) = f_*\mathcal End(\mathcal F_i)$.
It is a free $\mathcal O_{\Delta}$-module.
We note that there are no higher cohomologies because the $\mathcal F_i$ are pretilting.

Since $\mathcal F_i$ is flat over $\mathcal R_i$, there is an exact functor 
$\Phi_i: \text{mod}(\mathcal R_i) \to \text{coh}(\mathcal X)$ of $\mathcal O_{\Delta}$-linear abelian categories given by
$\Phi_i(\bullet) = \bullet \otimes_{\mathcal R_i} \mathcal F_i$.
We denote its derived functor also by $\Phi_i: D^b(\text{mod}(\mathcal R_i)) \to D^b(\text{coh}(\mathcal X))$.
It has a right adjoint functor $\Psi_i: D^b(\text{coh}(\mathcal X)) \to D^b(\text{mod}(\mathcal R_i))$ defined by 
$\Psi_i(\bullet) = R\text{Hom}(\mathcal F_i, \bullet)$.
Indeed we have $\text{Hom}_{\mathcal X}(a \otimes_{\mathcal R_i}^L \mathcal F_i, b) 
\cong \text{Hom}_{\mathcal R_i}(a, R\text{Hom}(\mathcal F_i,b))$. 
We have $\Psi_i\Phi_i \cong \text{Id}_{D^b(\text{mod}(\mathcal R_i))}$, and
the functors $\Phi_i$ are fully faithful, because
$\text{Hom}_{\mathcal R_i}(a,b) \cong \text{Hom}_{\mathcal R_i}(a,\Psi_i\Phi_i(b)) \cong 
\text{Hom}_{\mathcal X}(\Phi_i(a),\Phi_i(b))$.

Let $\mathcal C \subset D^b(\text{coh}(\mathcal X))$ be the right orthogonal complement of the $\mathcal F_i$,
i.e., $\mathcal C = \{x \in D^b(\text{coh}(\mathcal X)) \mid R\text{Hom}(\mathcal F_i, x) \cong 0 \,\, \forall i\}$. 
For $x \in D^b(\text{coh}(\mathcal X))$, we define the $x^{(i)} \in D^b(\text{coh}(\mathcal X))$ for $0 \le i \le 3$ 
inductively as follows.
Let $x^{(3)} = x$, and we define $x^{(i-1)}$ from $x^{(i)}$ by a distinguished triangle
\[
\Phi_i\Psi_i(x^{(i)}) \to x^{(i)} \to x^{(i-1)} \to \Phi_i\Psi_i(x^{(i)})[1].
\]
We claim that $\Psi_j(x^{(i)}) \cong 0$ for $j > i$.
Indeed, if we apply $\Psi_i$ to the above distinguished triangle, then we have $\Psi_i(x^{(i-1)}) \cong 0$.
If we apply $\Psi_j$ for $j > i$, then we have $\Psi_j(x^{(i-1)}) \cong \Psi_j(x^{(i)}) \cong 0$ because
$R\text{Hom}(\mathcal F_j, \mathcal F_i) \cong 0$.
Therefore we have $x^{(0)} \in \mathcal C$.

Let $t$ be a coordinate on $\Delta$.
We have an exact sequence $0 \to \mathcal O_{\mathcal X} \to \mathcal O_{\mathcal X} \to \mathcal O_X \to 0$,
where the first arrow is given by the multiplication of $t$.
Let us fix $c \in \mathcal C$.
By tensoring with the above exact sequence, we obtain a distinguished triangle
\[
c \to c \to c_0 \to c[1]
\]
where $c_0 = c \otimes_{\mathcal O_{\mathcal X}}^L \mathcal O_X \in D^b(\text{coh}(X))$.
Since $c_0 \in \mathcal C$, we have
\[
0 \cong R\text{Hom}_{\mathcal X}(\mathcal F_i, c_0) \cong R\text{Hom}_X(F_i, c_0)
\]
for all $i$.
Here we used $\omega_{\mathcal X} \otimes^L_{\mathcal O_{\mathcal X}} \mathcal O_X \cong \omega_X$, 
which is obtained by taking $\omega_{\mathcal X} \otimes^L_{\mathcal O_{\mathcal X}}$ to the above 
exact sequence, at the points where $\mathcal F_i$ is dual free.

Hence $c_0 \cong 0$, because the $F_i$ generate $D^b(\text{coh}(X))$.
It follows that $t: H^j(c) \to H^j(c)$ are bijective for all $j$.
Since the $H^j(c)$ are coherent sheaves on $\mathcal X$, their supports are closed subsets which do not
intersect $X$.
Since $c$ is a bounded complex, we conclude that $c \cong 0$ if we replace $\Delta$ by a smaller disk if necessary.

Let $\mathcal O_{\mathcal X}(1)$ be a divisorial sheaf on $\mathcal X$ such that 
$\mathcal O_{\mathcal X}(1) \otimes_{\mathcal O_{\mathcal X}} \mathcal O_X \cong \mathcal O_X(a_1a_2a_3)$
and $\mathcal O_{\mathcal X}(1) \otimes_{\mathcal O_{\mathcal X}} \mathcal O_Y \cong \mathcal O_{\mathbf P^2}(1)$.
We set $c_k = \mathcal O_{\mathcal X}(k)^{(0)} \in \mathcal C$ for $k = -2,-1,0$.
By shrinking $\Delta$ three times, we may assume that $c_k \cong 0$ for all $k$.

For any $y \in D^b(\text{coh}(Y))$, if
\[
0 \cong R\text{Hom}_{\mathcal X}(\mathcal O_{\mathcal X}(k), y) \cong R\text{Hom}_Y(\mathcal O_{\mathbf P^2}(k), y),
\]
then we have $y \cong 0$, because the $\mathcal O_{\mathbf P^2}(k)$ for $k=-2,-1,0$ 
generate $D^b(\text{coh}(\mathbf P^2))$.
Since $c_k \cong 0$ for all $k$, we deduce that, if 
$R\text{Hom}_{\mathcal X}(\mathcal F_i, y) \cong 0$ for all $i$, 
then $y \cong 0$.
Now assume that $R\text{Hom}_Y(E_i, y) \cong 0$ for all $i$.
Then we have
\[
R\text{Hom}_{\mathcal X}(\mathcal F_i, y) \cong R\text{Hom}_Y(E_i^{\oplus a_i}, y) \cong 0
\]
hence $y \cong 0$.
This completes the proof of the fulness of the $E_i$.
\end{proof}

Hacking-Prokhorov \cite{HaP} classified all normal projective surfaces $X$ having only quotient singularities 
such that $-K_X$ is ample, Picard number $\rho(X) = 1$, and that $X$ has a $\mathbf Q$-Gorenstein smoothing.
We consider the case of weighted projective planes.

\begin{Thm}
Let $X = \mathbf P(s_1a_1^2,s_2a_2^2,s_3a_3^2)$ be a weighted projective plane which has 
a $\mathbf Q$-Gorenstein smoothing $\mathcal X \to \Delta$ to a del Pezzo surface $Y$
as classified in \cite{HaP} Theorem 4.1, where $a_1,a_2,a_3$ are Cartier indexes of the canonical divisor 
$K_{\mathcal X}$ at the singular points.
Let $D^b(X) = \langle L_1,L_2,L_3 \rangle = \langle \overline F_1,\overline F_2,\overline F_3 \rangle$ 
be the semi-orthogonal decomposition explained in \S 3, 
where $\text{rank}(F_i) = s_ia_i^2$.
Then under the $\mathbf Q$-Gorenstein smoothing, 
$F_i$ deforms to a direct sum $\bigoplus_{j=1}^{s_i} E_{i,j}^{\oplus a_i}$ 
for a $3$-block full exceptional collection of vector bundles 
\[
(E_{1,1},\dots,E_{1,s_1};E_{2,1}, \dots, E_{2,s_2}; E_{3,1}, \dots, E_{3,s_3})
\]
on $Y$ such that $\text{rank}(E_{i,j}) = a_i$.
\end{Thm}

\begin{proof}
The proof similar to the previous theorem.
We note that the $3$ coordinate points of $X$ are automatically quotient singularities of the type $\frac 1{r^2s}(1,ars-1)$ or 
possibly $\frac 1s(1,s-1)$.
By the main theorem, $F_i$ deforms to $\bigoplus_{j=1}^{s_i} E_{i,j}^{\oplus a_i}$ for exceptional vector bundles $E_{i,j}$
which are mutually orthogonal.
We have again 
$\mathbb R\text{Hom}(E_{i,j},E_{i',j'}) \cong 0$ for $i > i'$.
The proof of the fullness is similar to that of Theorem \ref{Pa1a2a3}.
\end{proof}

We expect that these in the latter theorem coincide with Karpov-Nogin blocks (\cite{KN}).
More generally, toric surfaces in \cite{HaP} Theorem 4.1 can be treated similarly, because the 
construction of pretilting bundles works similarly as in the case of weighted projective planes.

%%%%%%%%%%%%%%%%%%%%%%%%%%%%%%%%%%%%%%%%%%
%%%%%%%%%%%%%%%%%%%%%%%%%%%%%%%%%%%%%%%%%%
%%%%%%%%%%%%%%%%%%%%%%%%%%%%%%%%%%%%%%%%%%

Graduate School of Mathematical Sciences, University of Tokyo,
Komaba, Meguro, Tokyo, 153-8914, Japan. 

kawamata@ms.u-tokyo.ac.jp

\end{document}